\documentclass[11pt,english]{article}
\usepackage[T1]{fontenc}
\usepackage[latin9]{inputenc}
\usepackage{refstyle}
\usepackage{float}
\usepackage{amsmath}
\usepackage{amsthm}
\usepackage{amssymb}
\usepackage{graphicx}
\usepackage{wasysym}

\makeatletter

\AtBeginDocument{\providecommand\secref[1]{\ref{sec:#1}}}
\RS@ifundefined{subsecref}
{\newref{subsec}{name = \RSsectxt}}
{}
\RS@ifundefined{thmref}
{\def\RSthmtxt{theorem~}\newref{thm}{name = \RSthmtxt}}
{}
\RS@ifundefined{lemref}
{\def\RSlemtxt{lemma~}\newref{lem}{name = \RSlemtxt}}
{}

\theoremstyle{plain}
\newtheorem{thm}{\protect\theoremname}
\theoremstyle{definition}
\newtheorem{defn}[thm]{\protect\definitionname}
\theoremstyle{remark}
\newtheorem{rem}[thm]{\protect\remarkname}
\theoremstyle{plain}
\newtheorem{cor}[thm]{\protect\corollaryname}
\theoremstyle{plain}
\newtheorem{conjecture}[thm]{\protect\conjecturename}
\theoremstyle{plain}
\newtheorem{lem}[thm]{\protect\lemmaname}
\theoremstyle{definition}
\newtheorem{example}[thm]{\protect\examplename}

\usepackage{fullpage}
\usepackage{hyperref}
\hypersetup{
	colorlinks   = true, %Colours links instead of ugly boxes
	urlcolor     = red, %Colour for external hyperlinks
	linkcolor    = blue, %Colour of internal links
	citecolor   = blue %Colour of citations, could be ``red''
}
\date{}

\makeatother

\usepackage{babel}
\providecommand{\conjecturename}{Conjecture}
\providecommand{\corollaryname}{Corollary}
\providecommand{\definitionname}{Definition}
\providecommand{\examplename}{Example}
\providecommand{\lemmaname}{Lemma}
\providecommand{\remarkname}{Remark}
\providecommand{\theoremname}{Theorem}

\begin{document}
	\global\long\def\goinf{\rightarrow\infty}
	\global\long\def\gozero{\rightarrow0}
	\global\long\def\bra{\langle}
	\global\long\def\ket{\rangle}
	\global\long\def\union{\cup}
	\global\long\def\intersect{\cap}
	\global\long\def\abs#1{\left|#1\right|}
	\global\long\def\norm#1{\left\Vert #1\right\Vert }
	\global\long\def\floor#1{\left\lfloor #1\right\rfloor }
	\global\long\def\ceil#1{\left\lceil #1\right\rceil }
	\global\long\def\expect{\mathbb{E}}
	\global\long\def\e{\mathbb{E}}
	\global\long\def\r{\mathbb{R}}
	\global\long\def\n{\mathbb{N}}
	\global\long\def\q{\mathbb{Q}}
	\global\long\def\c{\mathbb{C}}
	\global\long\def\z{\mathbb{Z}}
	\global\long\def\grad{\nabla}
	\global\long\def\t{^{\prime}}
	\global\long\def\all{\forall}
	\global\long\def\eps{\varepsilon}
	\global\long\def\quadvar#1{V_{2}^{\pi}\left(#1\right)}
	\global\long\def\cal#1{\mathcal{#1}}
	\global\long\def\cross{\times}
	\global\long\def\del{\nabla}
	\global\long\def\parx#1{\frac{\partial#1}{\partial x}}
	\global\long\def\pary#1{\frac{\partial#1}{\partial y}}
	\global\long\def\parz#1{\frac{\partial#1}{\partial z}}
	\global\long\def\part#1{\frac{\partial#1}{\partial t}}
	\global\long\def\partheta#1{\frac{\partial#1}{\partial\theta}}
	\global\long\def\parr#1{\frac{\partial#1}{\partial r}}
	\global\long\def\curl{\nabla\times}
	\global\long\def\rotor{\nabla\times}
	\global\long\def\one{\mathbf{1}}
	\global\long\def\Hom{\text{Hom}}
	\global\long\def\pr#1{\text{Pr}\left[#1\right]}
	\global\long\def\almost{\mathbf{\approx}}
	\global\long\def\tr{\text{Tr}}
	\global\long\def\var{\text{Var}}
	\global\long\def\onenorm#1{\left\Vert #1\right\Vert _{1}}
	\global\long\def\twonorm#1{\left\Vert #1\right\Vert _{2}}
	\global\long\def\Inj{\mathfrak{Inj}}
	\global\long\def\inj{\mathsf{inj}}
	
	\global\long\def\g{\mathfrak{\cal G}}
	\global\long\def\f{\mathfrak{\cal F}}

\title{Exponential random graphs behave like mixtures of stochastic block
	models}

\author{Ronen Eldan\thanks{Weizmann Institute of Science. Email: ronen.eldan@weizmann.ac.il.
		Partially supported by the Israel Science Foundation, grant 715/16.} ~and Renan Gross\thanks{Weizmann Institute of Science. Email: renan.gross@weizmann.ac.il}}
\maketitle

\begin{abstract}
We study the behavior of exponential random graphs in both the sparse and the dense regime. We show that exponential random graphs are approximate mixtures of graphs with independent edges whose probability matrices are critical points of an associated functional, thereby satisfying a certain matrix equation. In the dense regime, every solution to this equation is close to a block matrix, concluding that the exponential random graph behaves roughly like a mixture of stochastic block models. We also show existence and uniqueness of solutions to this equation for several families of exponential random graphs, including the case where the subgraphs are counted with positive weights and the case where all weights are small in absolute value. In particular, this generalizes some of the results in a paper by Chatterjee and Diaconis from the dense regime to the sparse regime and strengthens their bounds from the cut-metric to the one-metric. 
\end{abstract}

\tableofcontents{}
	
\section{Introduction}

With the emergent realization that large networks abound in science
(e.g metabolic networks), technology (e.g the internet), and everyday
life (e.g social networks), there has been widespread interest in
probabilistic models which capture the behavior of real life networks. 

The simplest random graph is the Erd\H{o}s-R\'{e}nyi $G\left(N,p\right)$
model of graphs with independent edges. While this model is well understood,
real networks often exhibit dependencies between the edges: For example,
in a social network, if two people have many mutual friends, it is
more likely that they themselves are friends. 

A natural and well studied model which captures edge dependencies
is the exponential random graph model, denoted here by $G_{N}^{f}$.
In this model, the probability to obtain a graph $G$ on $N$ vertices

\begin{equation}
\pr{G_{N}^{f}=G}=\exp\left(f\left(G\right)\right)/Z,\label{eq:definition_of_exponential_graph}
\end{equation}
where $f$ is a real functional on graphs called the ``Hamiltonian''
and $Z$ is a normalizing constant. Typically, $f$ is a ``subgraph-counting
function'' of the form 
\[
f\left(G\right)=\sum_{i=1}^{\ell}\beta_{i} N \left( H_i, G\right ),
\]
where the function $N\left(H_i, G\right)$ counts how many times the graph $H_{i}$
appears as a subgraph of $G$. The parameters $\beta_{i}$ are called
``weights'', and may be either positive or negative. For a review
of exponential graphs, see the papers in \cite{fienberg2010_I,fienberg2010_II}.

Despite the simple definition of this distribution, many basic aspects
about its behavior are far from being well-understood. For example,
there is at present no known explicit formula for the normalizing
constant. 

One of the first rigorous papers on the topic is due to Bhamidi, Bresler,
and Sly \cite{bbs2011}, which analyzes the mixing of the associated
Glauber dynamics in the case that subgraphs are counted with positive
weights, and gives a sufficient condition on those weights (referred
to as the ``high temperature regime'') under which any finite collection
of edges are asymptotically independent.

Another significant advance towards understanding the dense case was
done in a paper of Chatterjee and Diaconis \cite{chatterjee_diaconis_2013},
based on the technology developed in \cite{chatterjee_varadhan_2011},
which uses graph limit theory. They associate the normalizing constant
with a variational problem, showing that every exponential graph distribution
is close to the minimizing set of some functional on the space of
graphons. Further, if the Hamiltonian of this distribution counts
subgraphs only positively, then under the cut-metric the exponential
random graph is close to a $G\left(N,p\right)$ graph. 
In \cite{kenyon_et_al_bipodal_2018} and \cite{kenyon_et_al_multipodal_2014}, the graphon
framework also served the investigation of a similar problem, that of 
computing the asymptotic structure of graphs with constrained densities
of subgraphs.

More recently, in \cite{2016_eldan_gaussian_width}, it was shown
that an exponential graph is close in expectation to a mixture of
independent graphs. Unfortunately, this result gives no information
about the structure of those independent graphs.

\paragraph*{Our contributions}

In this work, we take one further step towards a better understanding
of exponential random graphs. We strengthen the existing results in
the following three ways:
\begin{enumerate}
\item We characterize the structure of the independent graphs of the mixture
model in \cite{2016_eldan_gaussian_width} by showing that the elements
of the mixture approximately obey a certain fixed point equation.
In particular, we show that under certain conditions, exponential
random graphs behave like mixtures of so-called stochastic block models.
\item We strengthen the results of both \cite{chatterjee_diaconis_2013}
and \cite{bbs2011} by characterizing the graph structure in terms
of the one-norm. This norm induces a stronger metric than the cut-metric
on the space of graphons, and gives some information about the nature
of dependence between the edges and other aspects which are not captured
by the cut-metric. 
\item Our characterization is meaningful not only in the dense regime, but
also in a limited range of sparse graphs as well. In particular, several
of our results hold for an edge density $p$ which depends polynomially
on $N$, e.g $p\geq N^{-c}$ for some $c>0$.
\end{enumerate}
The following is an overview of our main theorems. An independent
graph is a random graph whose edges are independent Bernoulli random
variables. Denote by $X$ the expected adjacency matrix of such a
graph. In \textbf{Theorem} \ref{thm:main_theorem}, we show that for
every subgraph-counting function $f$, the corresponding exponential
graph behaves like a mixture of independent graphs whose associated
expectations satisfy 
\[
\norm{X-\left(\one+\tanh\left(\grad f\left(X\right)\right)\right)/2}_{1}=o\left(N^{2}\right),
\]
where $\one$ is the matrix with zero on the diagonal and whose off-diagonal
entries are $1$, the $\tanh$ is applied entrywise, and $\onenorm X=\sum_{i,j}\abs{X_{ij}}$
is the one-norm. Using this result, we then characterize our mixtures
in three different settings:
\begin{enumerate}
\item \textbf{Theorem} \ref{thm:general_block_theorem} shows that every
subgraph-counting exponential random graph is $o\left(N^{2}\right)$
close to a mixture of stochastic block models with a small number
of blocks. 
\item \textbf{Theorem} \ref{thm:positive_uniqueness} roughly shows that
if the subgraphs are counted only with positive weights, then there
exists a constant matrix $X_{c}$ so that for every mixture element
$X$, $\norm{X-X_{c}}_{1}=o\left(N^{2}\right).$ Thus, the graph behaves
like $G\left(N,p\right)$.
\item \textbf{Theorem } \ref{thm:small_weights} shows that if the absolute values 
of the weights $\beta$ are small enough, then there exists a constant matrix $X_{c}$ so that
for every mixture element $X$, $\onenorm{X-X_{c}}=o\left(N^{2}\right).$
\end{enumerate}

\section{Background and notation}

Throughout the entire paper, $N>0$ is an integer that represents
the number of vertices and $n={N \choose 2}$ represents the number
of possible edges in an $N$ vertex simple graph. For two vertices $v$ and $u$ in a graph, $v \sim u$ denotes that $v$ is adjacent to $u$. We denote the discrete
hypercube by $\cal C_{n}=\left\{ 0,1\right\} ^{n}$ and the continuous
hypercube by $\overline{\cal C_{n}}=\left[0,1\right]^{n}$.

For ease of notation, we identify the vectors $\cal C_{n}$ with the
family of symmetric matrices of size $N\times N$ where the diagonal
entries are $0$ and the above diagonal entries are $0$ or $1$.
Such matrices correspond to simple graphs: For $X\in\cal C_{n}$,
the vertex $i$ is connected to vertex $j$ if and only if $X_{ij}=1$.
We therefore also identify the vector $X$ with the graph it represents.
For two graphs $G,G'$ whose corresponding vectors are $X,Y$, we
use the notation $\norm{G-G'}_{1}$ for $\onenorm{X-Y}$.

This view extends also to vectors $X\in\overline{\cal C_{n}}$, by
identifying with $X$ the weighted graph whose edge weights are $\left(X\right)_{ij}$.

Thus, any function acting on a vector $X\in\overline{\cal C_{n}}$
can also be seen as a function acting on a symmetric $N\times N$
matrix with $0$ diagonal or on a weighted graph on $N$ vertices,
and vice versa. 

We denote by $\one$ the matrix with zero on the diagonal and whose
off-diagonal entries are $1$.

\subsection{Subgraph counting functions\label{subsec:background_subgraph_counting_functions}}
\begin{defn}[Injective homomorphism density]
Let $G$ be a simple graph on $N$ vertices and let $H$ be a simple graph on $m$ vertices. Denote by $\Inj\left(H,G\right)$ the set
of injective homomorphisms from $H$ to $G$, that is, the set of functions $\phi:V\left(H\right)\rightarrow V\left(G\right)$ such
that if $x,y\in H$ and $x\sim y$, then $\phi\left(x\right)\sim\phi\left(y\right)$, and if $\phi\left(x\right)=\phi\left(y\right)$, then $x=y$. Denote the number of such homomorphisms by $\inj\left(H,G\right)=\abs{\Inj\left(H,G\right)}$. The ``\emph{injective homomorphism density}'' of $H$ is defined as 
\[
t\left(H,G\right)=\frac{\inj\left(H,G\right)}{N\left(N-1\right)\cdot\ldots\cdot\left(N-m+1\right)}.
\]
\end{defn}
\begin{defn}[Subgraph-counting function]
Let $\ell,N>0$ be integers. Let $H_{1},\ldots H_{\ell}$ be finite
simple graphs and $\beta_{1},\ldots,\beta_{\ell}$ be real numbers.
The functional $f$ on simple graphs with $N$ vertices defined by
\begin{equation}
f\left(G\right)=N\left(N-1\right)\sum_{i=1}^{\ell}\beta_{i}t\left(H_{i},G\right)\label{eq:definition_of_subgraph_counting}
\end{equation}
is called a ``\emph{subgraph-counting function}''.

As we will see below (in Section \ref{sec:proof_of_main_theorem})
the normalization $N\left(N-1\right)$ is natural since under this
normalization, the typical values of $f$ are of the same order as
the entropy of the graph.
\end{defn}

\begin{rem}
Subgraph counting functions are sometimes defined not by injective
homomorphisms but by all \emph{general} homomorphisms, denoted by
$\Hom\left(H,G\right)$. For our purposes, however, it is more convenient
to use injective homomorphisms to count subgraphs. The difference
between the injective homomorphism density and the general homomorphism
density is asymptotically small, so this distinction will not matter
in asymptotic calculations, and our results are equally valid for
general homomorphism densities. See \cite[Section 5.2.1--5.2.3]{large_networks_and_graph_limits}
for more details on such distinctions.
\end{rem}

Depending on both the weights and the subgraphs that are counted,
when using a subgraph-counting function as the Hamiltonian of an exponential
random graph, the resulting graph can be either sparse or dense. For
example, suppose that for a graph $G=\left(V,E\right)$ we define
\[
f\left(G\right)=\abs E\log\frac{p}{1-p}
\]
for some $p\in\left(0,1\right)$. Then 
\begin{align*}
\exp\left(f\left(G\right)\right)=\exp\left(\abs E\log\frac{p}{1-p}\right) & =p^{\abs E}\left(1-p\right)^{-\abs E}.
\end{align*}
The normalizing constant in this case is just $Z=\left(1-p\right)^{{N \choose 2}}$,
so that 
\begin{equation}
\pr{G_{N}^{f}=G}=p^{\abs E}\left(1-p\right)^{{N \choose 2}-\abs E}.\label{eq:sparse_exponential_graphs}
\end{equation}
This is exactly the $G\left(N,p\right)$ distribution, and if $p\rightarrow0$
when $N\goinf$ we obtain a sparse graph. 

\begin{defn}[Discrete gradient, Lipschitz constant] \label{def:discrete_gradient}

Let $f:\cal C_{n}\rightarrow\r$ be a real function on the Boolean
hypercube. The discrete derivative of $f$ at coordinate $i$ is defined as
\[
\partial_{i}f\left(Y\right)=\frac{1}{2}\left(f\left(Y_{1},\ldots,Y_{i-1},1,Y_{i+1},\ldots Y_{n}\right)-f\left(Y_{1},\ldots,Y_{i-1},0,Y_{i+1},\ldots Y_{n}\right)\right).
\]
With this we define both the the discrete gradient:
\[
\grad f\left(Y\right)=\left(\partial_{1}f\left(Y\right),\ldots,\partial_{n}f\left(Y\right)\right),
\]
and the Lipschitz constant of $f$: 
\[
\text{Lip}\left(f\right)=\max_{i\in\left[n\right],Y\in\cal C_{n}}\abs{\partial_{i}f\left(Y\right)}.
\]
Note that subgraph-counting functions and their gradients were originally
defined on simple graphs, or, alternatively, on vectors in $\cal C_{n}$.
However, they can be naturally extended to weighted graphs, or, alternatively,
to vectors in $\overline{\cal C_{n}}$, in the following way.

For a simple graph $G$, let $X$ be its adjacency matrix. A subgraph-counting
function $f$ that counts only a single graph $H=\left(\left[m\right],E\right)$
has the form (this is a slight variation from \cite[Lemma 33]{2016_eldan_gaussian_width}):
\end{defn}

\begin{equation}
f\left(G\right)=\frac{\beta}{\left(N-2\right)\left(N-3\right)\ldots\left(N-m+1\right)}\sum_{\underset{\text{\ensuremath{q} has distinct elements}\text{ }}{q\in\left[N\right]^{m}}}\prod_{\left(l,l'\right)\in E}X_{q_{l},q_{l'}}.\label{eq:how_to_calculate_f}
\end{equation}
Further, for an edge $e=\left\{ i,j\right\} $, the derivative satisfies

\begin{equation}
\partial f_{ij}\left(G\right)=\frac{\beta}{\left(N-2\right)\left(N-3\right)\ldots\left(N-m+1\right)}\sum_{\left(a,b\right)\in E}\sum_{\underset{q_{a}=i,q_{b}=j}{\underset{\text{\ensuremath{q} has distinct elements}}{q\in\left[N\right]^{m}}}}\prod_{\underset{\left\{ l,l'\right\} \neq\left\{ a,b\right\} }{\left(l,l'\right)\in E}}X_{q_{l},q_{l'}}.\label{eq:how_to_calculate_grad_f}
\end{equation}
As can be seen, both $f\left(G\right)$ and each entry of $\grad f\left(G\right)$
are just polynomials in the entries of $X$. This notation allows
us to extend $f$'s and $\grad f$'s domain to $\left[0,1\right]^{n}$,
and thus to weighted matrices and graphs. Note that since we count
injective homomorphisms and the entries of the vector $q$ in the
above calculation are distinct, the degree of each variable is either
$0$ or $1$. Further by equation (\ref{eq:how_to_calculate_grad_f}),
for every $x\in\left[0,1\right]$ we have that 
\[
\partial_{ij}f\left(x\one\right)=\beta\abs Ex^{\abs E-1}.
\]

\subsection{The variational approach }

To state the results of Chatterjee and Diaconis, we briefly present
some definitions from graph limit theory; for a detailed exposition,
see \cite[part 3]{large_networks_and_graph_limits}. Denote by $\cal W$
the space of all measurable functions $w:\left[0,1\right]^{2}\rightarrow\left[0,1\right]$,
and by $\tilde{\cal W}$ the space of equivalence classes of $\cal W$
under the equivalence relation $g\sim h\iff$there exists a measure
preserving bijection $\sigma:\left[0,1\right]\rightarrow\left[0,1\right]$
such that $g\left(x,y\right)=h\left(\sigma\left(x\right),\sigma\left(y\right)\right)=\left(\sigma h\right)\left(x,y\right)$.
The space $\tilde{W}$ is called the space of \emph{graphons}.

For every graph $G$ on $N$ vertices, it is possible to assign a
graphon $\tilde{G}$ by
\[
\tilde{G}\left(x,y\right)=\begin{cases}
1 & \ceil{xN}\sim\ceil{yN}\,\text{in \ensuremath{G}}\\
0 & o.w.
\end{cases}
\]
With this correspondence, every distribution on graphs induces a distribution
on graphons by the pushforward mapping. 

For any continuous bounded function $w:\left[0,1\right]^{2}\rightarrow\r$,
its cut-norm is defined as 

\[
\norm w_{\square}=\sup_{S,T\subseteq\left[0,1\right]}\abs{\int_{S\times T}w\left(x,y\right)dxdy}.
\]
This defines a metric on the space of graphons by $d_{\square}\left(\tilde{g},\tilde{h}\right)=\inf_{\sigma}\norm{\sigma g-h}_{\square}$,
where the infimum is taken over all measure preserving bijections
$\sigma$ as above. 

The results of Chatterjee and Diaconis can now be framed as follows. 
\begin{thm}[Theorem 3.2 in \cite{chatterjee_diaconis_2013}]
Let $f:\tilde{\cal W}\rightarrow\r$ be a continuous bounded functional.
Denote by $G_{N}^{f}$ the exponential random graph whose Hamiltonian
is $f\left(\tilde{G}\right)$. Then there exists a bounded continuous
functional $\varphi_{f}:\tilde{\cal W}\rightarrow\r$ which depends
on $f$ with the following property. Denote by $\tilde{F}^{*}$ the
set of graphons maximizing $\varphi_{f}$ . Then for any $\eta>0$
there exist $C,\gamma>0$ such that 
\[
\pr{d_{\square}\left(\tilde{G}_{N}^{f},\tilde{F}^{*}\right)>\eta}\leq Ce^{-N^{2}\gamma}.
\]
\end{thm}

As a corollary, they show the following result for subgraph counting
functions:
\begin{thm}[Theorem 4.2 in \cite{chatterjee_diaconis_2013}]
Assume that $H_{1}=K_{2}$ is the complete graph on two vertices
and that $\beta_{2},\ldots,\beta_{\ell}$ are all nonnegative. Then
the set of maximizers of $\varphi_{f}$ consists of a finite set of
constant graphons. Further, 
\[
\min_{\tilde{u}\in\tilde{F}^{*}}d_{\square}\left(\tilde{G}_{N}^{f},\tilde{u}\right)\rightarrow0\,\,\,\text{in probability as \ensuremath{N\goinf.}}
\]
\end{thm}

In other words, the exponential random graph $G_{N}^{f}$ is close
in the cut-distance to a distribution of Erd\H{o}s-R\'{e}nyi graphs $G\left(N,p\right)$
where $p$ is picked randomly from some probability distribution.

In a later paper, Chatterjee and Dembo \cite{chatterjee_dembo_2014}
derived a variational framework which yields nontrivial estimates
in the sparse regime. However, that framework does not seem to give
strong enough bounds on the partition function in order to charaterize
the associated distribution. 

\subsection{Mixture models}

In this paper, we are interested in approximating exponential random
graphs by mixtures of independent graphs. The following definitions
will be central to our results.
\begin{defn}[$\rho$-mixtures]
For $\vec{p}\in\left[0,1\right]^{{N \choose 2}}$, denote by $G\left(N,\vec{p}\right)$
the random graph with independent edges such that the edge $i \sim j$
appears with probability $\vec{p}_{ij}$. Let $\rho$ be a measure
on $\left[0,1\right]^{{N \choose 2}}$. We define the random vector
$G\left(N,\rho\right)$ by
\[
\pr{G\left(N,\rho\right)=G}=\int\pr{G\left(N,\vec{p}\right)=G}d\rho\left(\vec{p}\right).
\]
We say that $G\left(N,\rho\right)$ is a \emph{$\rho$-mixture}.
\end{defn}

\begin{defn}[Approximate mixture decomposition]
Let $\delta>0$ and let $\rho$ be a measure on $\left[0,1\right]^{{N \choose 2}}$.
A random graph $G$ is called a \emph{$\left(\rho,\delta\right)$-mixture
}if there exists a coupling between $G\left(N,\rho\right)$ and $G$
such that 
\[
\e\onenorm{G\left(N,\rho\right)-G}\leq\delta n.
\]

\end{defn}

A complementary result, given in \cite{2016_eldan_gaussian_width}
roughly states that an exponential random graph $G$ is close to a
$\left(\rho,o\left(1\right)\right)$-mixture in a way that most of
the entropy comes from the individual $G\left(N,\vec{p}\right)'s$ rather
than from the mixture.

For a random variable $X$ with law $\nu$, we define the entropy of $X$ as 
$$
\mathrm{Ent}(X) = \int -\log(\nu(x)) d\nu.
$$

\begin{thm}[Theorem 9 in \cite{2016_eldan_gaussian_width}]
\label{thm:eldan_mixture_theorem}For any positive integers $N,\ell$,
finite simple graphs $H_{1},\ldots,H_{\ell},$ real numbers $\beta_{1},\ldots,\beta_{\ell}$
and $\eps\in\left(0,1/2\right)$, the exponential graph defined in
\ref{eq:definition_of_exponential_graph}, is a $\left(\rho,\delta\right)$-mixture,
and such that 
\[
\delta\leq\frac{34n^{-1/12}}{\eps^{1/3}}\left(\sum_{i=1}^{\ell}\abs{\beta_{i}}\abs{E\left(H_{i}\right)}\right)^{1/3}
\]
with

\[
\mathrm{Ent}\left(G\left(N,\rho\right)\right)\leq\int\mathrm{Ent}\left(G\left(N,\vec{p}\right)\right)d\rho\left(\vec{p}\right)+\eps{N \choose 2}.
\]

\end{thm}

\section{Results}

The results of this paper are based on the following technical statement
which is an application of the framework in \cite{eldan_gross_gibbs_distribution}.
This result gives a characterization of the measure $\rho$ described
above: With high probability with respect to $\rho$, the vector $\vec{p}$
is nearly a critical point of a certain functional associated with
$f$. In order to formulate this result, let us make some notation.

For every subgraph-counting function $f$ of the form (\ref{eq:definition_of_subgraph_counting}),
define the constant 
\begin{align*}
C_{\boldsymbol{\beta}} & =\max\left\{ 12\sum_{i=1}^{\ell}\abs{\beta_{i}}\abs{E\left(H_{i}\right)}^{2},2\right\} .
\end{align*}
Remark that $C_{\boldsymbol{\beta}}$ depends only on the graph counting
parameters, barring $N$. Denote by $\cal X_{f}$ the set 
\begin{equation}
\cal X_{f}=\left\{ X\in\left[0,1\right]^{n}:\onenorm{X-\left(\one+\tanh\left(\grad f\left(X\right)\right)\right)/2}\leq 5000C_{\boldsymbol{\beta}}^2 n^{15/16}\right\} ,\label{eq:main_theorem_inequality}
\end{equation}
with the $\tanh$ applied entrywise to the entries of $\grad f\left(X\right)$.
\begin{thm}[Product decomposition of exponential random graphs]
\label{thm:main_theorem}Let $f$ be a subgraph counting function.
There exists a measure $\rho$ on $\left[0,1\right]^n$ (which depends on $n$ and on $f$) such that $G_{n}^{f}$ is a $\left(\rho,80\frac{C_{\boldsymbol{\beta}}}{n^{1/16}}\right)$-mixture with
\[
\rho\left(\cal X_{f}\right)\geq1-80\frac{C_{\boldsymbol{\beta}}}{n^{1/16}}.
\]
\end{thm}

In other words, almost all the mass of the mixture resides on random
graphs whose adjacency matrices $X$ almost satisfy the fixed point
equation

\begin{equation}
X=\frac{\one+\tanh\left(\grad f\left(X\right)\right)}{2}.\label{eq:fixed_point_equation}
\end{equation}

\begin{rem}
	In fact, more is known about the structure of the measure $\rho$. Following the notation in \cite{eldan_gross_gibbs_distribution}, for a vector $\theta \in \r^n$, the tilt $\tau_\theta \nu$ of a distribution $\nu$ is defined by 
	$$ \frac{d(\tau_\theta \nu)}{d\nu}(y) = \frac{e^{\langle \theta, y \rangle}}{\int_{\mathcal{C}_n}e^{\langle \theta, z \rangle}d\nu}.$$
	As it turns out, the measure $\rho$ in Theorem \ref{thm:main_theorem} is composed of small tilts, i.e., there exists a measure $m$ on $\r^n$ supported on small vectors $\theta$ such that $\rho$ is the pushforward of $m$ under the map $\theta \mapsto \mathbb{E}_{X \sim \tau_\theta \nu}\left[X\right]$. For more details, see \cite{eldan_gross_gibbs_distribution}.
\end{rem}

\begin{rem}
One can check that the solutions of the fixed point equation are exactly
the critical points of the functional $f\left(X\right)+H\left(X\right)$
where $H\left(X\right)=\sum_{i<j}X_{ij}\log X_{ij}+\left(1-X_{ij}\right)\log\left(1-X_{ij}\right)$
is the entropy of $X$. This is a variant of the functional that arises
in the variational problem in \cite{chatterjee_diaconis_2013}. 
\end{rem}

As described in \cite{eldan_gross_gibbs_distribution}, the solutions
to the equation $X=\left(\one+\tanh\left(\grad f\left(X\right)\right)\right)/2$
are critical points of a certain functional. Comparing our result
to Theorem 3.2 in \cite{chatterjee_diaconis_2013}: The latter shows
that the exponential graphs are close to global \emph{maxima} of the
variational problem, while the former only shows that it is close
to critical points; however, it gives a stronger, distributional description
and works beyond the dense regime.

Our first main result shows that in the dense regime, the matrices
obtained by Theorem \ref{thm:main_theorem} are close to matrices
that can be decomposed into a small number of blocks, defined as follows:

\begin{defn}[Stochastic block model]
Let $N,k>0$ be positive integers. A symmetric matrix $X\in\r^{N\times N}$
is called a ``\emph{block matrix}'' with $k$ communities, if there
exists a symmetric matrix $P\in\r^{k\times k}$ and a partition of
the indices $1,\ldots,N$ into $k$ disjoint sets $V_{1},\ldots,V_{k}$
such that for $i\in V_{\ell_{1}}$ and $j\in V_{\ell_{2}}$ with $\ell_1, \ell_2 \in \left[k \right]$,
\[
X_{ij}=P_{\ell_{1},\ell_{2}}.
\]
The sets $V_{1},\ldots,V_{k}$ are called the ``\emph{communities}''
of $X$. A random graph with independent edges whose expected adjacency
matrix is a block matrix is called a ``\emph{stochastic block model}''.
\end{defn}

\begin{thm}[Small number of communities for counting functions]
\label{thm:general_block_theorem}Let $0<\delta<1$ and let $f$
be a subgraph-counting function. Then there exists a constant $C_{\delta}>0$
(which depends on $\delta,$ the subgraphs $H_{i}$, and their weights
$\beta_{i}$ but is otherwise independent of $N$) such that for any
$X\in\cal X_{f}$, there exists a block matrix $X^{*}$ with no more
than $C_{\delta}$ communities such that 
\[
\onenorm{X-X^{*}}\leq\delta n+5000C_{\boldsymbol{\beta}}^2n^{15/16}.
\]
\end{thm}

One can derive an explicit expression for the constant $C_{\delta}$,
which is in general exponential in $1/\delta^{2}$. The explicit dependence
in the case of triangle-counting functions is derived in the proof. 

Theorems \ref{thm:main_theorem} and \ref{thm:general_block_theorem}
combined give the following corollary:
\begin{cor}
For any finite set of graphs $H_{1},\ldots,H_{\ell}$, constants $\beta_{1},\ldots,\beta_{\ell}$
and any constant $\delta>0$ there exists a constant $C_{\delta}$
such that the following holds. For every $N$, there exists a measure
$\rho$ supported on block matrices with at most $C_{\delta}$ communities
such that if $G_{N}^{f}$ is the exponential random graph with the
Hamiltonian $f\left(g\right)=N\left(N-1\right)\sum_{i=1}^{\ell}\beta_{i}t\left(H_{i},g\right)$
then there is a coupling between $G_{N}^{f}$ and $G\left(N,\rho\right)$
which satisfies
\[
\e\onenorm{G_{N}^{f}-G\left(n,\rho\right)}\leq\delta{N \choose 2}.
\]
\end{cor}

We conjecture that Theorem \ref{thm:general_block_theorem} can be
strengthened as follows:
\begin{conjecture}
Let $f$ be a subgraph-counting function. Then there is a constant
$c$ independent of $N$ (but dependent on the weights $\beta_{i}$)
such that every $X\in\cal X_{f}$ is $o\left(n\right)$-close to a
block matrix with no more than $c$ communities.
\end{conjecture}

Our second main result regarding the characterization of exponential
graphs applies to subgraph-counting functions with positive weights.
Its statement remains nontrivial for graphs with polynomially small
density, for some range of exponents, as will be demonstrated in Example
\ref{exa:sprasity_example}.

Following the notation of \cite{bbs2011}, we define $\varphi_{\boldsymbol{\beta}}:\left[0,1\right]\to\r$
by 
\[
\varphi_{\boldsymbol{\beta}}\left(x\right)=\frac{1+\tanh\left(\sum_{i=1}^{\ell}\beta_{i}\abs{E(H_{i})}x^{\abs{E(H_{i})}-1}\right)}{2}.
\]
Note that $\varphi_{\boldsymbol{\beta}}\left(x\right)$ is equal to
any off-diagonal entry of the constant matrix $\left(\one+\tanh\left(\grad f\left(x\one\right)\right)\right)/2$.
If the equation $x=\varphi_{\boldsymbol{\beta}}\left(x\right)$ has
a unique fixed point $x_{0}$, define the constant $D_{\boldsymbol{\beta}}=\sup_{\underset{x\neq x_{0}}{x\in\left[0,1\right]}}\frac{\abs{\varphi_{\boldsymbol{\beta}}\left(x\right)-x_{0}}}{\abs{x-x_{0}}}$.

The following simple lemma gives a useful bound on $D_{\boldsymbol{\beta}}$;
we present it without proof.
\begin{lem}
\label{lem:simple_constant_solution_D_bound}~
\begin{enumerate}
\item \label{enu:weak_existence}There exists an $x_{0}\in\left[0,1\right]$
such that $x_{0}=\varphi_{\boldsymbol{\beta}}\left(x_{0}\right)$.
Hence there always exists a constant solution $X_{c}=x_{0}\one$ to
the fixed point equation (\ref{eq:fixed_point_equation}).
\item \label{enu:high_temperature}Assume that $\varphi_{\boldsymbol{\beta}}\left(x\right)$
is increasing. If the solution $x_{0}$ is unique and $\varphi_{\boldsymbol{\beta}}'\left(x_{0}\right)<1$,
then $D_{\boldsymbol{\beta}}<1$.
\end{enumerate}
\end{lem}

The condition in item (\ref{enu:high_temperature}) in the above lemma
is referred to in \cite{bbs2011} as the \emph{high temperature regime}.
\begin{thm}[Positive weights]
\label{thm:positive_uniqueness}Let $N>3$ be an integer. Let $H_{1},\ldots,H_{\ell}$
be graphs, let $\alpha\in\r$ and $\beta_{1},\ldots,\beta_{\ell}\in\r$
be real numbers and let $f$ be a subgraph-counting function 
\[
f\left(X\right)=\alpha\inj\left(K_{2},X\right)+N\left(N-1\right)\sum_{i=1}^{\ell}\beta_{i}t\left(H_{i},X\right)
\]
where $K_{2}$ is the complete graph on 2 vertices. Assume that $\beta_{i}\geq0$
are positive for all $i$, that the equation $x=\varphi_{\boldsymbol{\beta}}\left(x\right)$
has a unique solution $x_{0}$ and that $D_{\boldsymbol{\beta}}<1$.
Then for any $X\in\cal X_{f}$ and any $0<\lambda<1$,
\begin{equation}
\onenorm{X-x_{0}\one}\leq\lambda n+10000C_{\boldsymbol{\beta}}^3\lambda^{\frac{\log C_{\boldsymbol{\beta}}}{\log D_{\boldsymbol{\beta}}}}n^{15/16}.\label{eq:positive_bound}
\end{equation}
In particular, for any constants $C_{\boldsymbol{\beta}}$ and $D_{\boldsymbol{\beta}}$
, there exists constants $0<\gamma<1/16$ and $Q>0$ such that
\begin{equation}
\onenorm{X-x_{0}\one}\leq Q\cdot n^{1-\gamma}.\label{eq:optimized_lambda}
\end{equation}
\end{thm}

Our third main result regarding the characterization of exponential graphs applies to subgraph-counting functions whose weights are small in absolute value: If all $\beta$'s are small enough, the only solution to equation (\ref{eq:fixed_point_equation}) is the trivial one.

\begin{thm}[Small weights]
\label{thm:small_weights}
Let $N>3$ be an integer. Let $H_{1},\ldots,H_{\ell}$ be graphs, let
$\alpha\in\r$ and $\beta_{1},\ldots,\beta_{\ell}\in\r$ be real numbers
and let $f$ be a subgraph-counting function 
\[
f\left(X\right)=\alpha\inj\left(K_{2},X\right)+N\left(N-1\right)\sum_{i=1}^{\ell}\beta_{i}t\left(H_{i},X\right)
\]
where $K_{2}$ is the complete graph on 2 vertices. Denote $m_{i}=\abs{E\left(H_{i}\right)}$
and define the sum
\[
S_{\boldsymbol{\beta}}=\sum_{i=1}^{\ell}\abs{\beta_{i}}{m_{i} \choose 2}.
\]
If $S_{\boldsymbol{\beta}}<1$, then the constant solution $X_{c}$
obtained from item (\ref{enu:weak_existence}) in Lemma \ref{lem:simple_constant_solution_D_bound} is the only solution to the fixed point equation (\ref{eq:fixed_point_equation}). Further, any $X\in\mathcal{X}_{f}$ satisfies
\[
\norm{X-X_{c}}_{1}\leq\frac{5000C_{\boldsymbol{\beta}}^2}{1-S_{\boldsymbol{\beta}}}n^{15/16}.
\]
\end{thm}

\begin{rem}
	
	One should compare Theorem \ref{thm:positive_uniqueness} and Theorem \ref{thm:small_weights} to Theorems 4.2 and 6.2 in \cite{chatterjee_diaconis_2013}, respectively. There, similar conditions (positive $\beta$'s or $S_{\boldsymbol{\beta}} < 1$) imply that the exponential random graph is close in the cut metric to a finite set of constant graphons.
\end{rem}

Finally, for the particular case of triangle-counts, it turns out that if $\beta<0$ is smaller than some universal
constant, there exists at least one non-trivial solution in the form of two blocks. 
\begin{thm}[Two block model]
\label{thm:two_block_model}
Let $N>3$ be an integer, let $\beta\in\r$,
and let $f\left(X\right)=\frac{\beta}{N-2}\inj\left(K_{3},X\right)$,
where $K_{3}$ is the triangle graph. There exists a $\beta_{0}<0$ such that if $\beta<\beta_{0}$, there is a solution to equation (\ref{eq:fixed_point_equation})
in the form of a block model with 2 communities. Specifically, the
$N$ vertices can be divided into two sets of equal size $U$ and
$W$, such that $X_{ij}=c_{1}$ if $\left(i,j\right)\in\left(U\times W\right)\union\left(W\times U\right)$,
and $X_{ij}=c_{2}$ if $\left(i,j\right)\in\left(U\times U\right)\union\left(W\times W\right)$
for $i\neq j$. Further, as $\beta\rightarrow-\infty$, $c_{1}\rightarrow\frac{1}{2}$
and $c_{2}\rightarrow0$.

\end{thm}

\begin{rem}[A remark on bounds and sparsity]	
 When considering subgraph counting functions, it is useful to think
of the special case that the $\beta_{i}$'s are constants independent
of $N$. In this case, the typical exponential graph will be dense,
and inequalities involving the one-norm of matrices will yield meaningful
information. However, letting the $\beta_{i}$'s depend explicitly on $N$ can lead to sparse graphs. The sparse case is typically harder analyze than the dense case, although there are some exact results in this regime (see e.g., \cite{yin_zhu_2017} where the partition function and two-edge correlations are derived for certain families of $\beta$'s). \\
Our theorems still hold true in the sparse regime, but for graphs which are too sparse they may only be trivially true. Consider
Theorem \ref{thm:positive_uniqueness} as an example. If the weights
$\beta_{i}$ are such that the expected number of edges in the exponential
graph is smaller than the error term $\inf_{\lambda\in\left(0,1\right)}\lambda n+1000C_{\boldsymbol{\beta}}^3\lambda^{\frac{\log C_{\boldsymbol{\beta}}}{\log D_{\boldsymbol{\beta}}}}n^{15/16}$,
then the weight matrix is trivially close to a constant matrix: Namely,
the zero matrix. In this case the theorem tells us nothing new. The
next example demonstrates that this is not always the case, and our
results can give meaningful information in the sparse regime.
\end{rem}

\begin{example}
\label{exa:sprasity_example}In this informal example, we give a sketch
for the case of triangle counts. Let $f$ be the function
\[
f\left(X\right)=\alpha\inj\left(K_{2},X\right)+\frac{\beta}{N-2}\inj\left(K_{3},X\right)
\]
where $\alpha=\frac{1}{2}\log\frac{p}{1-p}$ and $\frac{1}{200}\abs{\alpha} \leq \beta\leq\frac{1}{100}\abs{\alpha}$.
We will take $p=p\left(N\right)=n^{-c}$ for some $c>0$. This implies
that $\alpha\propto-\log N$ and $\beta\propto\log N$; thus $\alpha\to-\infty$
and $\beta\goinf$ as $N\goinf$. We expect the typical number of
edges in the resulting exponential graph to be $\Omega\left(np\right)$. 

It can be verified that for large enough $N$, there is only a single
solution to the equation $x=\varphi_{\boldsymbol{\beta}}\left(x\right)$;
denote it by $x_{0}$. Our first task is to calculate $D_{\boldsymbol{\beta}}$.
By its definition, it is always smaller than the maximum of the derivative
of $\varphi_{\boldsymbol{\beta}}\left(x\right)=\frac{1+\tanh\left(\alpha+3\beta x^{2}\right)}{2}$.
Thus, neglecting constants, 
\begin{align*}
D_{\boldsymbol{\beta}}\leq\max_{x\in\left[0,1\right]}\abs{\varphi'\left(x\right)} = & \max_{x\in\left[0,1\right]}\frac{3\beta x^{2}}{\cosh^{2}\left(\alpha+3\beta x^{2}\right)} \\
& \leq \frac{3\beta}{\cosh^{2}\left(\frac{1}{2}\alpha\right)} \\
& \apprle \abs{\alpha}e^{\alpha}.
\end{align*}
Hence for all $N$ large enough, we have $D_{\boldsymbol{\beta}}<1$,
and can apply Theorem \ref{thm:positive_uniqueness}: For any $X\in\cal X_{f}$,
we have
\[
\onenorm{X-x_{0}\one}\leq\lambda n+10000C_{\boldsymbol{\beta}}^3\lambda^{\frac{\log C_{\boldsymbol{\beta}}}{\log D_{\boldsymbol{\beta}}}}n^{15/16}.
\]
Now, since $D_{\boldsymbol{\beta}} \leq \abs{\alpha}e^\alpha < 1$, we have that $\abs{\log D_{\boldsymbol{\beta}}}\gtrsim \abs{\alpha}$, while $C_{\boldsymbol{\beta}}\approx\abs{\alpha}$,
so $\log C_{\beta}\approx\log\abs{\alpha}$; this gives
\[
\abs{\frac{\log C_{\boldsymbol{\beta}}}{\log D_{\boldsymbol{\beta}}}}\lesssim\frac{\log\abs{\log p}}{\abs{\log p}}\approx\frac{\log\log n}{c\log n}.
\]
Set $\lambda=n^{-1}=e^{-\log n}$. Then
\[
\lambda^{\frac{\log C_{\boldsymbol{\beta}}}{\log D}}\approx e^{-\log n\frac{\log\log n}{-c\log n}}=e^{c'\log\log n}=\left(\log n\right)^{c'}.
\]
Since we want the error term to be smaller than the number of edges,
then ignoring all logarithmic terms (i.e those coming from $\lambda$ and $C_{\boldsymbol{\beta}}$), we require the following inequality to hold:
\[
np\apprge n^{15/16}.
\]
This indeed allows a polynomial dependence between $p$ and $n$.
For any $p$ satisfying
\[
p\apprge n^{-1/16},
\]
we conclude that there exists a constant $p'$ and a coupling between
$G(n,p')$ and $G_{N}^{f}$ such that
\[
\e\|G(n,p')-G_{N}^{f}\|_{1}=o(np).
\]
\end{example}

\subsection{Open questions and further directions}
\begin{itemize}
\item Theorems \ref{thm:positive_uniqueness} and \ref{thm:small_weights}
show that in some cases, the random graphs in the mixture are close
to an actual fixed point of equation (\ref{eq:fixed_point_equation}).
It is natural to ask whether this is a general phenomenon. Let $X\in\cal X_{f}$
and denote by $S=\left\{ Y:Y=\left(\one+\tanh\left(\grad f\left(Y\right)\right)\right)/2\right\} $
the set of solutions to the fixed point equation (\ref{eq:fixed_point_equation}).
Is it true that 
\[
\inf_{Y\in S}\onenorm{X-Y}=o\left(n\right)?
\]
In other words, is it true that approximately-fixed points are approximately
fixed-points?
\item How quickly can the the parameter $\delta$ in Theorem \ref{thm:general_block_theorem}
approach $0$ while still keeping a meaningful bound? Can the theorem
be improved to obtain a polynomial dependence on $N$?
\item Can Theorem \ref{thm:general_block_theorem} be formulated in a meaningful
way for sparse exponential random graphs?
\item Lubetzky and Zhao proposed in \cite{lubetzky_zhao_dense} a variant
of subgraph-counting functions where the Hamiltonian is of the form
\[
f\left(G\right)=N\left(N-1\right)\left(\sum_{i=1}^{\ell}\beta_{i}t\left(H_{i},G\right)^{\alpha_{i}}\right)
\]
for some $\alpha_{1},\ldots\alpha_{\ell}>0$. Theorem 14 in \cite{eldan_gross_gibbs_distribution}
implies that this modified Hamiltonian also breaks up into a mixture
of product measures. What are the components of this mixture? Is there
a criterion on the exponents $\alpha_{i}$ that enables / ensures
symmetry-breaking?
\item The fixed point equation $X=\left(\one+\tanh\left(\grad f\left(X\right)\right)\right)/2$
corresponds to the critical points of a variational problem. Classify
these critical points; is it true that they are all maxima? If not,
how does the mass of $\rho$ distribute among the different types?
In particular, is the mass always distributed on \emph{global} maxima?
\item Show that for the case of triangle counts, every solution to the exact
fixed point equation $X=\left(\one+\tanh\left(\frac{\beta}{N-2}X^{2}\right)\right)/2$
is close to a stochastic block model with two communities. In other
words, show an ``only if'' condition for Theorem \ref{thm:two_block_model}.
\end{itemize}

\subsection*{Organization}

The rest of this paper is organized as follows. The proof of Theorem
\ref{thm:main_theorem} is given in \secref{proof_of_main_theorem}.
In \secref{general_block_model}, we prove the block model Theorem
\ref{thm:general_block_theorem}; we first show the proof for triangle-counting
functions, and then generalize it to arbitrary counting functions.
Finally, \secref{existence_and_positive_weghts}, \secref{small_weights} and 
\secref{two_block_model} are devoted to proving the existence and uniqueness of solutions of
the fixed point equation in some special cases, as described in Theorems
\ref{thm:positive_uniqueness}, \ref{thm:small_weights} and \ref{thm:two_block_model}.

\section{Proof of the mixture decomposition\label{sec:proof_of_main_theorem}}

The proof of Theorem \ref{thm:main_theorem} will follow as a corollary
from the main result of \cite{eldan_gross_gibbs_distribution}. In
order to formulate this result, we need the following definition.
\begin{defn}[Gaussian width, gradient complexity]
The \emph{Gaussian-width} of a set $K\subseteq\r^{n}$ is defined
as
\[
\boldsymbol{\mathrm{GW}}\left(K\right)=\e\left[\sup_{X\in K}\left\langle X,\Gamma\right\rangle \right]
\]
where $\Gamma\sim N\left(0,\textrm{Id}\right)$ is a standard Gaussian
vector in $\r^{n}$. For a function $f:\cal C_{n}\to\r$, the \emph{gradient
complexity} of $f$ is defined as
\[
\cal D\left(f\right)=\boldsymbol{\mathrm{GW}}\left(\left\{ \grad f\left(Y\right):Y\in\cal C_{n}\right\} \union\left\{ 0\right\} \right).
\]
\end{defn}

The main result of \cite{eldan_gross_gibbs_distribution} reads:
\begin{thm}[Theorem 9 in \cite{eldan_gross_gibbs_distribution} ]
\label{thm:main_gibbs_theorem}Let $n>0$, let $f:\cal C_{n}\to\r$, and let $X_{n}^{f}$ be a random vector given by the law
$$
\pr{X_{n}^{f} = X} = \exp(f(X)) / Z,
$$
where $Z$ is a normalizing constant. Denote
\begin{align*}
D & =\cal D\left(f\right)\\
L_{1} & =\max\left\{ 1,\text{Lip}\left(f\right)\right\} \\
L_{2} & =\max\left\{ 1,\max_{X\neq Y\in\cal C_{n}}\frac{\onenorm{\grad f\left(X\right)-\grad f\left(Y\right)}}{\onenorm{X-Y}}\right\} .
\end{align*}
Denote by $\cal X_{f}$ the set 
\[
\cal X_{f}=\left\{ X\in\overline{C_{n}}:\onenorm{X-\frac{\one+\tanh\left(\grad f\left(X\right)\right)}{2}}\leq5000L_{1}L_{2}^{3/4}D^{1/4}n^{3/4}\right\} 
\]
where $\one$ is the $N\times N$ matrix with zero on the diagonal
and whose off-diagonal entries are $1$, $\grad f\left(X\right)$
is extrapolated to $\overline{\cal C_{n}}$ by equation (\ref{eq:how_to_calculate_grad_f})
and with the $\tanh$ applied entrywise to the entries of $\grad f\left(X\right)$.
Then $X_{n}^{f}$ is a $\left(\rho,80\frac{D^{1/4}}{n^{1/4}}\right)$-mixture
such that 
\[
\rho\left(\cal X_{f}\right)\geq1-80\frac{D^{1/4}}{n^{1/4}}.
\]
\end{thm}

We will prove Theorem \ref{thm:main_theorem} by applying the above
theorem; this requires giving bounds on $\cal D\left(f\right)$, $\text{Lip}\left(f\right)$
and $\max\frac{\onenorm{\grad f\left(x\right)-\grad f\left(y\right)}}{\onenorm{x-y}}$.
We bound the latter two quantities in the following three lemmas. 

For a vector $X\in\cal C_{n}$, denote by $X_{j}^{+}$ the vector
$X_{j}^{+}=\left(X_{1},X_{2},\ldots,X_{j-1},1,X_{j+1},\ldots,X_{n}\right)$,
and by $X_{j}^{-}$ the vector $X_{j}^{-}=\left(X_{1},X_{2},\ldots,X_{j-1},0,X_{j+1},\ldots,X_{n}\right)$.
In terms of graphs, $X_{j}^{+}$ is the graph $X$ with the edge at
index $j$ added (if it is not already there), while $X_{j}^{-}$
is the graph $X$ with the edge at index $j$ removed.

The first lemma states that such subgraph-counting functions have
bounded Lipschitz constants. 
\begin{lem}
\label{lem:lipschitz_of_subgraph_hom}Let $f$ be a subgraph-counting
function of the form (\ref{eq:definition_of_subgraph_counting}).
Then for every $X\in\cal C_{n}$ and for every index $j$, $\abs{\partial_{j}f\left(X\right)}\leq\sum_{i=1}^{\ell}\abs{\beta_{i}}\abs{E\left(H_{i}\right)}.$
In other words, $f$ is $\sum_{i=1}^{\ell}\abs{\beta_{i}}\abs{E\left(H_{i}\right)}$-Lipschitz. 
\end{lem}

\begin{proof}
By definition, for any graph $H$, 
\begin{align*}
\partial_{j}\inj\left(H,X\right) & =\frac{\inj\left(H,X_{j}^{+}\right)-\inj\left(H,X_{j}^{-}\right)}{2}.
\end{align*}
The graphs $X_{j}^{+}$ and $X_{j}^{-}$ differ by only one edge,
which we call $e$. Now look at $\inj\left(H,X_{j}^{+}\right)-\inj\left(H,X_{j}^{-}\right).$
All homomorphisms which do not send at least one edge of $H$ into
the edge $e$ cancel out in this sum. Hence it is equal to 
\[
\#\left\{ \phi\in\Inj\left(H,X_{j}^{+}\right):e\in E\left(\phi\left(H\right)\right)\right\} .
\]
To bound the number of such homomorphisms, we construct them as follows:
first map one of the edges of $H$ to the edge $e$, and then injectively
map the remaining vertices of $H$ to vertices of $G$. There are
$2\abs{E\left(H\right)}$ ways to do the former and $\left(N-2\right)\left(N-3\right)\ldots\left(N-m+1\right)$
ways to do the latter, so overall:
\begin{equation}
\partial_{j}\inj\left(H,X\right)=\frac{\inj\left(H,X_{j}^{+}\right)-\inj\left(H,X_{j}^{-}\right)}{2}\leq\abs{E\left(H\right)}\left(N-2\right)\left(N-3\right)\ldots\left(N-m+1\right).\label{eq:homm_lipschitz}
\end{equation}
This means that 
\begin{align*}
\abs{\partial_{j}f\left(X\right)} & =\abs{\partial_{i}N\left(N-1\right)\sum_{i=1}^{\ell}\frac{\beta_{i}\inj\left(H_{i},X\right)}{N\left(N-1\right)\ldots\left(N-m+1\right)}}\\
\left(\text{triangle ineq}.\right) & \leq\sum_{i=1}^{\ell}\abs{\beta_{i}}\abs{\partial_{i}\frac{\inj\left(H_{i},X\right)}{\left(N-2\right)\ldots\left(N-m+1\right)}}\\
\left(\text{by\,(\ref{eq:homm_lipschitz})}\right) & \leq\sum_{i=1}^{\ell}\abs{\beta_{i}}\abs{E\left(H_{i}\right)}
\end{align*}
as needed.
\end{proof}
The second lemma tells us that that if $X$ and $Y$ differ by only
one index, then $\grad f\left(X\right)$ and $\grad f\left(Y\right)$
are close to each other.
\begin{lem}
\label{lem:similar_vectors_intersect_or_not}Let $f$ be a subgraph-counting
function. Let $X,Y\in\cal C_{n}$ be two vectors that differ only
in a single index $k$. Let $j$ be an index, $e_{j}$ be the edge
that corresponds to index $j$, and $e_{k}$ be the edge that corresponds
to index $k$. If $e_{j}$ and $e_{k}$ share a common vertex, then
\[
\abs{\partial_{j}f\left(X\right)-\partial_{j}f\left(Y\right)}\leq\sum_{i=1}^{\ell}\frac{2\abs{\beta_{i}}\abs{E\left(H_{i}\right)}^{2}}{\sqrt{n}}.
\]
If $e_{j}$ and $e_{k}$ do not share a common vertex, then 
\[
\abs{\partial_{j}f\left(X\right)-\partial_{j}f\left(Y\right)}\leq\sum_{i=1}^{\ell}\frac{6\abs{\beta_{i}}\abs{E\left(H_{i}\right)}^{2}}{n}.
\]
\end{lem}

\begin{proof}
Assume without loss of generality that $X_{k}=1$ while $Y_{k}=0$.
This means that $X$ contains the edge $e_{k}$ while $Y$ does not.
Then for every graph $H$, 
\begin{align*}
\partial_{j}\inj\left(H,X\right)-\partial_{j}\inj\left(H,Y\right) & =\frac{\inj\left(H,X_{j}^{+}\right)-\inj\left(H,X_{j}^{-}\right)}{2}-\frac{\inj\left(H,Y_{j}^{+}\right)-\inj\left(H,Y_{j}^{-}\right)}{2}.
\end{align*}
We can assume that $j\neq k$: If they were equal, then $X_{j}^{+}$
and $X_{j}^{-}$ would be equal to $Y_{j}^{+}$ and $Y_{j}^{-}$,
respectively, and the difference $\partial_{j}\inj\left(H,X\right)-\partial_{j}\inj\left(H,Y\right)$
would just be $0$. 

Similar to the proof of Lemma \ref{lem:lipschitz_of_subgraph_hom},
the first term $\inj\left(H,X_{j}^{+}\right)-\inj\left(H,X_{j}^{-}\right)$
counts the number of homomorphisms from $H$ to $X$ that map an edge
of $H$ into the edge $e_{j}$, while the second term $\inj\left(H,Y_{j}^{+}\right)-\inj\left(H,Y_{j}^{-}\right)$
counts the number of homomorphisms from $H$ to $Y$ that map an edge
of $H$ into the edge $e_{j}$. However, the homomorphisms in the
first term may map edges from $H$ into the edge $e_{k}$, while those
of the second term may not, since $e_{k}$ does not exist in $Y$.
Thus, their difference is equal to:
\[
\partial_{j}\inj\left(H,X\right)-\partial_{j}\inj\left(H,Y\right)=\frac{\#\left\{ \phi\in\Inj\left(H,X_{j}^{+}\right):\left\{ e_{j},e_{k}\right\} \subseteq E\left(\phi\left(H\right)\right)\right\} }{2}.
\]
To bound the number of such homomorphisms, we construct them as follows:
first map two of the edges of $H$ to the edges $e_{j}$ and $e_{k}$,
and then injectively map the remaining vertices of $H$ to vertices
of $G$. There are less than $\left(2\abs{E\left(H\right)}\right)^{2}$
ways to do the former. For the latter, it depends on whether $e_{j}$
and $e_{k}$ have a vertex in common. If they do not, then the edges
in $H$ mapping to $e_{j}$ and $e_{k}$ must also be disjoint, and
mapping them involves choosing $4$ vertices to map to the vertices
of $e_{j}$ and $e_{k}$. This gives $\left(N-4\right)\ldots\left(N-m+1\right)$
ways to map the remaining vertices of $H$. If $e_{j}$ and $e_{k}$
do have a vertex in common, then it is possible to map the corresponding
edges of $H$ by mapping only $3$ vertices to the vertices of $e_{j}$
and $e_{k}$. This gives $\left(N-3\right)\ldots\left(N-m+1\right)$
ways to map the remaining vertices of $H$.

So overall, we get that 
\begin{align*}
e_{j}\intersect e_{k} & =\emptyset\implies\partial_{j}\inj\left(H,X\right)-\partial_{j}\inj\left(H,Y\right)\leq2\abs{E\left(H\right)}^{2}\left(N-4\right)\ldots\left(N-m+1\right),\\
e_{j}\intersect e_{k} & \neq\emptyset\implies\partial_{j}\inj\left(H,X\right)-\partial_{j}\inj\left(H,Y\right)\leq2\abs{E\left(H\right)}^{2}\left(N-3\right)\ldots\left(N-m+1\right).
\end{align*}
This means that for $e_{j}\intersect e_{k}=\emptyset$, we get
\begin{align*}
\abs{\partial_{j}f\left(X\right)-\partial_{j}f\left(Y\right)} & =\abs{\partial_{j}N\left(N-1\right)\sum_{i=1}^{\ell}\beta_{i}\frac{\inj\left(H_{i},X\right)}{N\ldots\left(N-m+1\right)}-\partial_{j}N\left(N-1\right)\sum_{i=1}^{\ell}\beta_{i}\frac{\inj\left(H_{i},Y\right)}{N\ldots\left(N-m+1\right)}}\\
\left(\text{triangle ineq.}\right) & \leq\sum_{i=1}^{\ell}\frac{\abs{\beta_{i}}}{\left(N-2\right)\ldots\left(N-m+1\right)}\abs{\partial_{j}\inj\left(H_{i},X\right)-\partial_{j}\inj\left(H_{i},Y\right)}\\
 & \leq\sum_{i=1}^{\ell}\frac{\abs{\beta_{i}}}{\left(N-2\right)\ldots\left(N-m+1\right)}\left(2\abs{E\left(H_{i}\right)}^{2}\left(N-4\right)\ldots\left(N-m+1\right)\right)\\
 & =\sum_{i=1}^{\ell}\frac{2\abs{\beta_{i}}\abs{E\left(H_{i}\right)}^{2}}{\left(N-2\right)\left(N-3\right)}\leq\sum_{i=1}^{\ell}\frac{6\abs{\beta_{i}}\abs{E\left(H_{i}\right)}^{2}}{n},
\end{align*}
while for $e_{j}\intersect e_{k}\neq\emptyset$, we get 
\begin{align*}
\abs{\partial_{j}f\left(X\right)-\partial_{j}f\left(Y\right)} & \leq\sum_{i=1}^{\ell}\frac{\abs{\beta_{i}}}{\left(N-2\right)\ldots\left(N-m+1\right)}\left(2\abs{E\left(H_{i}\right)}^{2}\left(N-3\right)\ldots\left(N-m+1\right)\right)\\
 & =\sum_{i=1}^{\ell}\frac{2\abs{\beta_{i}}\abs{E\left(H_{i}\right)}^{2}}{N-2}\leq\sum_{i=1}^{\ell}\frac{2\abs{\beta_{i}}\abs{E\left(H_{i}\right)}^{2}}{\sqrt{n}}
\end{align*}
as needed.
\end{proof}
This result can be generalized to arbitrary $X,Y$, giving us
a bound for the one-norm $\onenorm{\grad f\left(X\right)-\grad f\left(Y\right)}$.
\begin{lem}
\label{lem:grad_is_lipschitz}Let $f$ be a subgraph-counting function.
Let $X,Y\in\overline{\cal C_{n}}$ be two vectors. Then
\[
\onenorm{\grad f\left(X\right)-\grad f\left(Y\right)}\leq C\onenorm{X-Y},
\]
where $C=12\sum_{i=1}^{\ell}\abs{\beta_{i}}\abs{E\left(H_{i}\right)}^{2}$.
\end{lem}

\begin{proof}
First, assume that $X$ and $Y$ differ only in single coordinate
$k$. Then for each coordinate $j$, either the edge $e_{j}$ intersects
with $e_{k}$ or not. 
Holding all other coordinates fixed, $\nabla f$ is linear as a function of the $k$-th coordinate. Then using Lemma \ref{lem:similar_vectors_intersect_or_not},
we can write:

\begin{align*}
\onenorm{\grad f\left(X\right)-\grad f\left(Y\right)} & =\sum_{j=1}^{n}\abs{\partial_{j}f\left(X\right)-\partial_{j}f\left(Y\right)}\abs{X_{k}-Y_{k}}\\
 & \leq\sum_{i=1}^{\ell}\sum_{j=1}^{n}\left(\one_{e_{j}\intersect e_{k}=\emptyset}\frac{6}{n}+\one_{e_{j}\intersect e_{k}\neq\emptyset}\frac{2}{\sqrt{n}}\right)\abs{\beta_{i}}\abs{E\left(H_{i}\right)}^{2}\abs{X_{k}-Y_{k}}.
\end{align*}
The edge $e_{k}$ can intersect at most $N$ different edges at each
of its endpoints, so the number of indices $j$ for which $e_{j}\intersect e_{k}\neq\emptyset$
is bounded by $2N\leq2\sqrt{2n}$. The number of indices $j$ for
which $e_{j}\intersect e_{k}=\emptyset$ is trivially bounded by $n$,
giving 
\begin{align*}
\onenorm{\grad f\left(X\right)-\grad f\left(Y\right)} & \leq\sum_{i=1}^{\ell}\sum_{j=1}^{n}\left(\one_{e_{j}\intersect e_{k}=\emptyset}\frac{6}{n}+\one_{e_{j}\intersect e_{k}\neq\emptyset}\frac{2}{\sqrt{n}}\right)\abs{\beta_{i}}\abs{E\left(H_{i}\right)}^{2}\abs{X_{k}-Y_{k}}\\
 & \leq\sum_{i=1}^{\ell}\left(6\frac{n}{n}+\frac{2\cdot2\sqrt{2}\sqrt{n}}{\sqrt{n}}\right)\abs{\beta_{i}}\abs{E\left(H_{i}\right)}^{2}\abs{X_{k}-Y_{k}}\\
 & \leq12\sum_{i=1}^{\ell}\abs{\beta_{i}}\abs{E\left(H_{i}\right)}^{2}\abs{X_{k}-Y_{k}}.
\end{align*}
The above reasoning is valid for $X$ and $Y$ which differ by one
coordinate; by the triangle inequality we achieve the desired result
for arbitrary $X,Y\in\overline{\cal C_{n}}$. 
\end{proof}
\begin{proof}[Proof of Theorem \ref{thm:main_theorem}]
By \cite[Section 5]{2016_eldan_gaussian_width}, the Gaussian-width
of the image of $\grad f$ is bounded by 
\[
\cal D\left(f\right)\leq\sum_{i}\abs{\beta}\abs{E(H_{i})}N^{3/2}\leq C_{\boldsymbol{\beta}}n^{3/4}.
\]
By Lemma \ref{lem:lipschitz_of_subgraph_hom}, $\text{Lip}\left(f\right)\leq C_{\boldsymbol{\beta}}$,
and by Lemma \ref{lem:grad_is_lipschitz}, 
\[
\max_{X,Y\in\cal C_{n}}\frac{\onenorm{\grad f\left(X\right)-\grad f\left(Y\right)}}{\onenorm{X-Y}}\leq C_{\boldsymbol{\beta}}
\]
as well. Plugging these bounds into Theorem \ref{thm:main_gibbs_theorem}
we obtain the desired results. 
\end{proof}

\section{Approximate block model for the dense regime \label{sec:general_block_model}}

In this section we prove Theorem \ref{thm:general_block_theorem}.
It will be instructive to first prove the theorem for triangle-counting
functions, as this case is simple and gives easy-to-calculate bounds.
The same techniques will then be used to give a sketch of the proof
for general subgraph-counting functions. 

The proof technique uses random orthogonal projections in order to
perform some of the calculations in a low-dimensional space. For this
we will need the following results concerning concentration of measure
of orthogonal random projections:
\begin{lem}[Orthogonal projections preserve distance. Due to \cite{dasgupta_gupta_johnson_lindenstrauss},
page 62]
\label{lem:johnson_linden_norms}Let $0<\delta<1$, let $d,k>0$
be positive integers, let $\pi:\r^{d}\rightarrow\r^{k}$ be an orthogonal
projection into a uniformly random $k$ dimensional subspace, and
let $g:\r^{d}\rightarrow\r^{k}$ be defined as $g\left(v\right)=\sqrt{\frac{d}{k}}\pi\left(v\right)$.
Then for any vector $v\in\r^{d}$, 
\begin{align*}
\pr{\left(1-\delta\right)\norm v^{2}\leq\norm{g\left(v\right)}^{2}\leq\left(1+\delta\right)\norm v^{2}} & \leq2e^{-k\left(\delta^{2}/2-\delta^{3}/3\right)/2}.
\end{align*}
\end{lem}

From this lemma about the magnitude of vectors, it is possible to
obtain similar bounds on the scalar product between two vectors:
\begin{lem}[Preserving scalar products]
\label{lem:johnson_linden_scalar}Let $0<\delta<1$, let $d,k>0$
be positive integers, and let $g:\r^{d}\rightarrow\r^{k}$ be a linear
transformation. Let $u,v\in\r^{d}$ be two vectors of norm smaller
than $1$ such that $\left(1-\delta\right)\norm{u\pm v}^{2}\leq\norm{g\left(u\pm v\right)}^{2}\leq\left(1+\delta\right)\norm{u\pm v}^{2}$.
Then 
\[
\abs{\left\langle g\left(v_{1}\right),g\left(v_{2}\right)\right\rangle -\left\langle v_{1},v_{2}\right\rangle }\leq2\delta.
\]
\end{lem}

The proof is postponed to the appendix.

\subsection{Counting triangles\label{subsec:block_model_triangle_count}}
\begin{proof}[Proof of Theorem \ref{thm:general_block_theorem} (for the case of triangle-counting functions)]
Let $N$ be a positive integer, let $\alpha,\beta\in\r$ be real
numbers, and let be $f$ of the form
\[
f\left(X\right)=\alpha\inj\left(K_{2},X\right)+\frac{\beta}{N-2}\inj\left(K_{3},X\right)
\]
where $K_{2}$ is the complete graph on two vertices and $K_{3}$
is the triangle graph. Let $X\in\cal X_{f}$. It can be verified by
direct calculation that
\[
f\left(X\right)=\alpha\tr\left(X^{2}\right)+\frac{\beta}{N-2}\text{Tr}\left(X^{3}\right)
\]
and 
\begin{equation}
\grad f\left(X\right)=\alpha\one+\frac{3\beta}{N-2}\overline{X^{2}},\label{eq:gradient_for_triangles}
\end{equation}
where $\overline{X^{2}}$ is the matrix with zero on the diagonal
and whose off-diagonal entries are those of $X^{2}$. We then have
by Theorem \ref{thm:main_theorem} that 
\begin{equation}
\onenorm{X-\frac{\one+\tanh\left(\alpha\one+\frac{3\beta}{N-2}\overline{X^{2}}\right)}{2}}\leq5000C_{\boldsymbol{\beta}}^2n^{15/16}.\label{eq:approximate_fixed_point_01_form}
\end{equation}
We proceed to show that the term $\frac{3\beta}{N-2}\overline{X^{2}}$
is close to a block matrix with a small number of communities. This
is done roughly as follows: Each entry in the matrix $\frac{3\beta}{N-2}\overline{X^{2}}$
can be written as the scalar product of two vectors in $\r^{N}$;
namely, the column vectors of $\sqrt{\frac{3\beta}{N-2}}X$. It is
possible to project these vectors into a low-dimensional space, so
that their scalar products are almost preserved. This low dimensional
projection can then be rounded to a $\delta$-net, whose size depends
only on $\delta$ and on the dimension. Thus if the dimension is small,
then the $\delta$-net is small. The matrix $\frac{3\beta}{N-2}\overline{X^{2}}$
can then be approximated by scalar products of elements from the $\delta$-net,
and each element in the net defines a community. Applying $\tanh$
entrywise, adding the constant $\one$ and dividing by $2$ does not
change the block model parameters, implying that $X$ itself is close
to a block matrix.

Denote by $v_{i}$ the $i$-th column of $X$ multiplied by $1/\sqrt{N}$,
so that 
\[
\left(v_{i}\right)_{j}=\frac{1}{\sqrt{N}}X_{ij}.
\]
Since all the entries of $X$ are in $\left[0,1\right]$, each $v_{i}$
lies within the unit ball:
\begin{equation}
\norm{v_{i}}^{2}=\frac{1}{N}\sum_{j=1}^{N}X^2_{ij}\leq1.\label{eq:norm_one_for_indicators}
\end{equation}
Let $f$ be a triangle-counting function, and assume that $\beta=1$.
Then for two distinct vertices $i$ and $j$, the derivative $\partial_{ij}f$
is equal to 
\[
\partial_{ij}f\left(X\right)=\frac{N}{N-2}\left\langle v_{i},v_{j}\right\rangle =\frac{N}{N-2}\frac{1}{N}\sum_{k}X_{ik}X_{kj}=\left(\frac{1}{N-2}\overline{X^{2}}\right)_{ij}.
\]
This is because the difference between $\inj\left(K_{3},G\right)$
with $G$ containing the edge $ij$ and $\inj\left(K_{3},G\right)$
where $G$ does not contain the edge $ij$ is exactly the sum of weights
of all the triangles of the form $ijk$ for $k=1,\ldots,n$. 

Let $k>0$ be a positive integer to be chosen later, let $U\subseteq\r^{N}$
be a uniformly random subspace of dimension $k$, and denote by $\pi:\r^{N}\rightarrow U$
an orthogonal projection from $\r^{N}$ into $U$. Let $g:\r^{N}\rightarrow U$
be defined as $g\left(v\right)=\sqrt{\frac{N}{k}}\pi\left(v\right)$.
For every two indices $i\neq j$, denote 
\[
\cal B_{ij}=\left\{ \left(1-\delta\right)\norm x^{2}\leq\norm{g\left(x\right)}^{2}\leq\left(1+\delta\right)\norm x^{2}\text{ for \ensuremath{x\in\left\{ v_{i},v_{j},v_{i}+v_{j},v_{i}-v_{j}\right\} }}\right\} ,
\]
the event that $g$ almost preserves the squared norm of both of the
original vectors $v_{i}$ and $v_{j}$ and of their sum and difference
$v_{i}+v_{j}$ and $v_{i}-v_{j}$. By Lemma \ref{lem:johnson_linden_norms},
the probability for $\cal B_{ij}$ to occur is at least 
\begin{equation}
\pr{\cal B_{ij}}\geq1-8e^{-k\left(\delta^{2}/2-\delta^{3}/3\right)/2}.\label{eq:good_event_probability}
\end{equation}
Under this event, since $\delta<1$, both $g\left(v_{i}\right)$ and
$g\left(v_{j}\right)$ are contained inside a ball of radius $2$
around the origin. Further, by Lemma \ref{lem:johnson_linden_scalar},
the scalar product between $v_{i}$ and $v_{j}$ is also almost preserved:
\begin{equation}
\abs{\left\langle g\left(v_{i}\right),g\left(v_{j}\right)\right\rangle -\left\langle v_{i},v_{j}\right\rangle }\leq2\delta.\label{eq:almost_there_scalar}
\end{equation}
Let $T$ be a $\delta$-net of the ball of radius $2$ around the
origin in $k$ dimensions. By \cite[lemma 2.6]{milman_schechtman_asymptotic_theory},
there exists such a net of size smaller than $\left(1+4/\delta\right)^{k+1}$.
For every vertex $i$, denote by $w_{i}=\text{argmin}_{w\in T}\norm{g\left(v_{i}\right)-w}$
the vector in $T$ that is closest to $g\left(v_{i}\right)$, and
denote by $\Delta w_{i}=w_{i}-g\left(v_{i}\right)$ the difference
between the two. Then under $\cal B_{ij}$, since $g\left(v_{i}\right)$
is in the ball of radius $2$, the magnitude of the difference $\norm{\Delta w_{i}}$
is smaller than $\delta$. In this case, 
\begin{align*}
\abs{\left\langle w_{i},w_{j}\right\rangle -\left\langle g\left(v_{i}\right),g\left(v_{j}\right)\right\rangle } & =\abs{\left\langle g\left(v_{i}\right)+\Delta w_{i},g\left(v_{j}\right)+\Delta w_{j}\right\rangle -\left\langle g\left(v_{i}\right),g\left(v_{j}\right)\right\rangle }\\
 & =\abs{\left\langle g\left(v_{i}\right),\Delta w_{j}\right\rangle +\left\langle \Delta w_{i},g\left(v_{j}\right)\right\rangle +\left\langle \Delta w_{i},\Delta w_{j}\right\rangle }\\
\left(\text{since \ensuremath{\norm{\Delta w}\leq\delta}}\right) & \leq6\delta.
\end{align*}
Thus, under $\cal B_{ij}$ and together with equation (\ref{eq:almost_there_scalar}),
we almost surely have that 
\[
\abs{\left\langle w_{i},w_{j}\right\rangle -\left\langle v_{i},v_{j}\right\rangle }\leq8\delta.
\]
Denote by $\tilde{X}$ the matrix defined by $\left(\tilde{X}\right)_{ij}=\left\langle w_{i},w_{j}\right\rangle $
for $i\neq j$ and with $0$ on the diagonal. It is clear that the
matrix $\tilde{X}$ is a block matrix, with the communities in correspondence
with the elements of the $\delta$-net $T$; hence there are no more
than $\left(1+4/\delta\right)^{k+1}$ communities in $\tilde{X}$.

The expected value of the one-norm between $\frac{1}{N}\overline{X^{2}}$
and $\tilde{X}$ is
\begin{align}
\e\norm{\frac{1}{N}\overline{X^{2}}-\tilde{X}}_{1} & =\e\sum_{i,j}\abs{\frac{1}{N}\left(\overline{X^{2}}\right)_{ij}-\left(\tilde{X}\right)_{ij}}\nonumber \\
 & =\sum_{i\neq j}\e\abs{\left\langle v_{i},v_{j}\right\rangle -\left\langle w_{i},w_{j}\right\rangle }.\label{eq:one_norm_expectation_bound_triangles}
\end{align}
Each expectation term of the form $\e\abs{\left\langle v_{i},v_{j}\right\rangle -\left\langle w_{i},w_{j}\right\rangle }$
can be controlled by conditioning on the event $\cal B_{ij}$. Keeping
in mind that in the general case $\abs{\left\langle v_{i},v_{j}\right\rangle -\left\langle w_{i},w_{j}\right\rangle }\leq5$
since the norm of $v_{i}$ and $v_{j}$ is bounded by $1$ and the
norm of $w_{i}$ and $w_{j}$ is bounded by $2$, we can bound the
expectation by 
\begin{align*}
\e\abs{\left\langle v_{i},v_{j}\right\rangle -\left\langle w_{i},w_{j}\right\rangle } & =\e\left[\abs{\left\langle v_{i},v_{j}\right\rangle -\left\langle w_{i},w_{j}\right\rangle }\mid\cal B_{ij}\right]\pr{\cal B_{ij}}\\
 & \,\,\,\,\,\,+\e\left[\abs{\left\langle v_{i},v_{j}\right\rangle -\left\langle w_{i},w_{j}\right\rangle }\mid\lnot\cal B_{ij}\right]\pr{\lnot\cal B_{ij}}\\
 & \leq8\delta\cdot1+5\cdot8e^{-k\left(\delta^{2}/2-\delta^{3}/3\right)/2}.
\end{align*}
Choosing $k=\left\lceil 2\log\left(1/\delta\right)\left(\delta^{2}/2-\delta^{3}/3\right)^{-1}\right\rceil $,
we have
\[
\e\abs{\left\langle v_{i},v_{j}\right\rangle -\left\langle w_{i},w_{j}\right\rangle }\leq48\delta.
\]
Plugging this into equation (\ref{eq:one_norm_expectation_bound_triangles}),
we obtain the bound 
\[
\e\norm{\frac{1}{N}\overline{X^{2}}-\tilde{X}}_{1}\leq48\delta n.
\]
Hence, there exists a block matrix $\hat{X}$ with no more than $\left(1+4/\delta\right)^{k+1}$
communities such that
\[
\norm{\frac{1}{N}\overline{X^{2}}-\hat{X}}_{1}\leq48\delta n.
\]
Multiplying both sides by $3\beta N/\left(N-2\right)$, we have that
\[
\norm{\frac{3\beta}{N-2}\overline{X^{2}}-\frac{3\beta N}{N-2}\hat{X}}_{1}\leq144\frac{N}{N-2}\beta\delta n\leq450\beta\delta n.
\]
Note that the function $\tanh$ is contracting; that is,
\begin{equation}
\abs{\tanh\left(x\right)-\tanh\left(y\right)}\leq\abs{x-y}.\label{eq:tanh_is_contracting}
\end{equation}
This gives implies that
\begin{align*}
\norm{\frac{\one+\tanh\left(\alpha\one+\frac{3\beta}{N-2}\overline{X^{2}}\right)}{2}-\frac{\one+\tanh\left(\alpha\one+\frac{3\beta N}{N-2}\hat{X}\right)}{2}}_{1} & =\frac{1}{2}\norm{\frac{3\beta}{N-2}\overline{X^{2}}-\frac{3\beta N}{N-2}\hat{X}}_{1}\\
 & \leq225\beta\delta n.
\end{align*}
Finally, by equation (\ref{eq:approximate_fixed_point_01_form}),
\[
\norm{X-\frac{\one+\tanh\left(\alpha\one+\frac{3\beta}{N}\overline{X^{2}}\right)}{2}}_{1}\leq5000C_{\boldsymbol{\beta}}^2n^{15/16},
\]
and so by the triangle inequality, denoting $X^{*}=\frac{\one+\tanh\left(\alpha\one+3\beta\hat{X}\right)}{2}$,
\[
\onenorm{X-X^{*}}\leq225\beta\delta n+5000C_{\boldsymbol{\beta}}^2n^{15/16}.
\]
\end{proof}

\subsection{Counting general subgraphs}

In this section we give a proof sketch of general form of Theorem
\ref{thm:general_block_theorem}. The proof relies on the same techniques
as those in the previous subsection, which gave block matrix bounds
for the specific case of triangles. 

Let $X\in\cal X_{f}$. The main argument in the previous proof was
as follows: For triangles, each entry in the gradient $\grad f\left(X\right)=\frac{3\beta}{N-2}\overline{X^{2}}$
was written as a scalar product between two vectors. These vectors
were then projected to a low dimensional space, yielding a block matrix
form. 

We will generalize the above procedure, and show that the gradient
$\grad f$ of any subgraph-counting function $f$ can be written as
a sum of scalar products of vectors: There exist an integer $S>0$,
a family of constants $c_{r}$, $r=1,\ldots,S$, and two families
of vectors $v_{i}^{r}$ and $u_{i}^{r}$ of norm smaller than $1$,
such that

\begin{equation}
\partial_{ij}f=\sum_{r=1}^{S}c_{r}\left\langle v_{i}^{r},u_{j}^{r}\right\rangle .\label{eq:general_block_good_form}
\end{equation}
The number of scalar products $S$ and the constants $c_{r}$ depend
on the subgraphs $H_{k}$ that $f$ counts and their weights $\beta_{k}$,
but do not grow explicitly with $N$. Repeating the reasoning in the
previous proof, these vectors can all be simultaneously projected
by an orthogonal projection $g$ to a low dimensional space, so that
\[
\partial_{ij}f\approx\sum_{r=1}^{S}c_{r}\left\langle g\left(v_{i}^{r}\right),g\left(u_{j}^{r}\right)\right\rangle .
\]
Taking a $\delta$-net of the sphere in the new space will give us
an approximation of these sums: For every $r$ we will obtain a block
matrix $W^{r}$ whose $ij$-th entry approximates the scalar product
$\left\langle g\left(v_{i}^{r}\right),g\left(u_{j}^{r}\right)\right\rangle $.
As before, the number of communities of $W^{r}$ will depend only
on $\delta$. Finally, since the sum of $S$ block matrices is also
a block matrix (albeit with a number of communities exponential in
$S$), the sum $\sum_{r=1}^{S}c_{r}\left\langle g\left(v_{i}^{r}\right),g\left(u_{j}^{r}\right)\right\rangle $
is itself a block matrix, with a number of communities that depends
only on the subgraphs $H_{k}$, their weights $\beta_{k}$, and on
$\delta$.

Let us now fill in some of the details for this proof sketch. Let
$H=\left(\left[m\right],E\left(H\right)\right)$ be a finite simple
graph on $m$ vertices with edge set $E\left(H\right)$. This simple
edge set can also be viewed as a directed edge set, with two directed
edges replacing every original simple edge: $D\left(H\right)=\bigcup_{\left\{ x,y\right\} \in E\left(H\right)}\left\{ \left(x,y\right),\left(y,x\right)\right\} $.
The essential part of the proof is showing that $\partial_{i,j}\inj\left(H,G\right)$
can be obtained by scalar products as above; the rest will follow
from linearity. 

Let $i$ be a vertex of $G$ and let $e=\left(x,y\right)\in D\left(H\right)$
be an oriented edge of $H$. Denote by $\Phi_{e}$ the set of all
injective maps from $H\backslash\left\{ x,y\right\} $ to $G$. The
vectors $v_{i}^{e}$ and $u_{i}^{e}$ will have one entry for every
function $\phi\in\Phi_{e}$. For $v_{i}^{e}$, the entry $v_{i}^{e}\left(\phi\right)$
contains the weight of edges from $i$ to the image $\phi\left(H\backslash\left\{ y\right\} \right)$,
times the square root of the weight of the image $\phi\left(H\backslash\left\{ x,y\right\} \right)$.
For $u_{i}^{e}$, the entry $u_{i}^{e}\left(\phi\right)$ contains
the weight of edges from $i$ to the image $\phi\left(H\backslash\left\{ x\right\} \right)$,
times the square root of the weight of the image $\phi\left(H\backslash\left\{ x,y\right\} \right)$.
More formally, for every $\phi\in\Phi_{e}$,
\begin{align*}
v_{i}^{e}\left(\phi\right) & =\prod_{a\,s.t\,\left\{ x,a\right\} \in E\left(H\backslash\left\{ y\right\} \right)}X_{i,\phi\left(a\right)}\prod_{\left\{ a,b\right\} \in E\left(H\backslash\left\{ x,y\right\} \right)}\sqrt{X_{\phi\left(a\right),\phi\left(b\right)}}\\
u_{i}^{e}\left(\phi\right) & =\prod_{a\,s.t\,\left\{ y,a\right\} \in E\left(H\backslash\left\{ x\right\} \right)}X_{i,\phi\left(a\right)}\prod_{\left\{ a,b\right\} \in E\left(H\backslash\left\{ x,y\right\} \right)}\sqrt{X_{\phi\left(a\right),\phi\left(b\right)}}.
\end{align*}
For two different vertices $i\neq j$, the scalar product between
two vectors becomes
\[
\left\langle v_{i}^{e},u_{j}^{e}\right\rangle =\sum_{\phi\in\Phi_{e}}\left(\prod_{\left\{ x,a\right\} \in E\left(H\backslash\left\{ y\right\} \right)}X_{i,\phi\left(a\right)}\prod_{\left\{ y,a\right\} \in E\left(H\backslash\left\{ x\right\} \right)}X_{j,\phi\left(a\right)}\prod_{\left\{ a,b\right\} \in E\left(H\backslash\left\{ x,y\right\} \right)}X_{\phi\left(a\right),\phi\left(b\right)}.\right)
\]
Let's inspect this scalar product. For each fixed $\phi$, the summand
is the edge weight of the image of the homomorphism $\psi:H\rightarrow G$,
where
\[
\psi\left(z\right)=\begin{cases}
i & z=x\\
j & z=y\\
\phi\left(z\right) & \text{o.w.}
\end{cases}.
\]
The mapping $\psi$ is in general not an injection: Although $\phi$
itself was chosen to be an injection, the function $\psi$ is not
one-to-one when $\phi\left(a\right)=i$ or $\phi\left(a\right)=j$
for some $a\in H$. But in this case, either $X_{i,\phi\left(a\right)}$
or $X_{j,\phi\left(a\right)}$ are $0$, since the diagonal entries
of $X$ are $0$. Thus, summing over all $\phi$ effectively means
summing over all injective mappings that send the particular (directed)
edge $\left(x,y\right)$ in $H$ to $\left(i,j\right)$ in $G$. By
the discussion in the proof of Lemma \ref{lem:lipschitz_of_subgraph_hom},
summing over all possible edges $e$ that can map to $\left(i,j\right)$
exactly gives the definition of the discrete derivative: 
\[
\partial_{ij}\inj\left(H,G\right)=\frac{1}{2}\sum_{e\in D\left(H\right)}\left\langle v_{i}^{e},u_{j}^{e}\right\rangle .
\]
The gradient of a subgraph-counting function that counts a single
subgraph $H$ with weight $\beta$ can then be written as 
\[
\partial_{ij}f=\frac{\beta}{2\left(N-2\right)\ldots\left(N-m+1\right)}\sum_{e\in D\left(H\right)}\left\langle v_{i}^{e},u_{j}^{e}\right\rangle .
\]
When we proved the theorem for the case of triangles, it was important
that the vectors were of unit length - this meant that the projection
was contained in a ball of radius $2$, and this is what allowed us
to take a $\delta$-net that did not depend on $N$. This is the case
here as well: Each entry of $v_{i}^{e}$ and $u_{i}^{e}$ is bounded
by $1$. Their norm is therefore bounded by the square root of the
number of entries, which is the number of injective mappings from
$H\backslash\left\{ x,y\right\} $ to $G$. Thus,
\[
\norm{v_{i}^{e}}^{2}\leq\abs{\Phi_{e}}<N\left(N-1\right)\ldots\left(N-m+3\right).
\]
This means that $v_{i}^{e}/\sqrt{N\cdot\ldots\cdot\left(N-m+3\right)}$
and $u_{i}^{e}/\sqrt{N\cdot\ldots\cdot\left(N-m+3\right)}$ have their
norm bounded by $1$.

Finally, for the case of general subgraph-counting functions that
count the subgraphs $H_{1},\ldots,H_{\ell}$ with weights $\beta_{1},\ldots,\beta_{\ell}$,
we have that 

\begin{align*}
\partial_{ij}f = & \sum_{k=1}^{\ell}\frac{N\left(N-1\right)}{\left(N-m_{k}+2\right)\left(N-m_{k}+1\right)} \cdot \\ 
& \left(\frac{\beta_{k}}{2}\sum_{r=1}^{2\abs{E\left(H_{k}\right)}}\left\langle \frac{v_{i}^{k,r}}{\sqrt{N\cdot\ldots\cdot\left(N-m_{k}+3\right)}},\frac{u_{j}^{k,r}}{\sqrt{N\cdot\ldots\cdot\left(N-m_{k}+3\right)}}\right\rangle \right).
\end{align*}
This shows that $\partial_{ij}f$ can indeed be written in the form
of equation (\ref{eq:general_block_good_form}). 

\section{Positive weights\label{sec:existence_and_positive_weghts}}

\subsection{The exact case\label{subsec:general_uniqueness}}

We would like to first give some intuition regarding the proof of
Theorem \ref{thm:positive_uniqueness}: We will show that if all the
weights $\beta_{i}$ are positive and if $x=\varphi\left(x\right)$
has a unique solution, then the fixed point equation $X=\left(\one+\tanh\left(\grad f\left(X\right)\right)\right)/2$
has a single solution $x_{0}\one$. The proof that any $X\in\cal X_{f}$
is close to $x_{0}\one$ will be more involved but analogous.

For clarity, we will assume that $f$ counts edges and triangles.
Let $\alpha,\beta\in\r$ with $\beta>0$, and let $f$ be of the form
\[
f\left(X\right)=\alpha\inj\left(K_{2},X\right)+\frac{\beta}{N-2}\inj\left(K_{3},X\right),
\]
where $K_{2}$ is an edge and $K_{3}$ is the triangle graph. Direct
calculation shows that $\grad f\left(X\right)=\alpha\one+\frac{3\beta}{N-2}\overline{X^{2}}$.
In terms of the adjacency matrix, the fixed point equation is then
\begin{equation}
X=\frac{\one+\tanh\left(\alpha\one+\frac{3\beta}{N-2}\overline{X^{2}}\right)}{2}.\label{eq:fixed_point_equation_with_edge_term}
\end{equation}
Let $X$ be a solution to equation (\ref{eq:fixed_point_equation_with_edge_term}).
Denote by $a$ the minimum off-diagonal entry of $X$ and by $b$
the maximum off-diagonal entry of $X$. For every index $i$ and $j$
with $i\neq j$ we have: 
\begin{align*}
\frac{3\beta}{N-2}\left(\overline{X^{2}}\right)_{ij} & =\frac{3\beta}{N-2}\sum_{k=1}^{N}X_{ik}X_{kj}.
\end{align*}
For $k=i$ and $k=j$, we have $X_{ii}=X_{jj}=0$. For all other indices
$k$, $X_{ik}\leq b$ by definition, so

\begin{align}
\frac{3\beta}{N-2}\left(\overline{X^{2}}\right)_{ij} & \leq3\beta b^{2}.\label{eq:smaller_than_b}
\end{align}
This is where the condition $\beta>0$ comes into play: The inequality
would have been reversed had $\beta$ been negative. The maximum element
of the right hand side of equation (\ref{eq:fixed_point_equation_with_edge_term})
is
\[
\max\frac{1+\tanh\left(\alpha\one+\frac{3\beta}{N-2}\overline{X^{2}}\right)}{2}\leq\frac{1+\tanh\left(\alpha\one+3\beta b^{2}\right)}{2}.
\]
Taking the maximum of both sides of equation (\ref{eq:fixed_point_equation_with_edge_term}),
we get
\[
b\leq\frac{1+\tanh\left(\alpha\one+3\beta b^{2}\right)}{2}.
\]
By similar argument, we get that 
\[
\frac{3\beta}{N-2}\left(\overline{X^{2}}\right)_{ij}\geq3\beta a^{2},
\]
and hence
\[
a\geq\frac{1+\tanh\left(\alpha\one+3\beta a^{2}\right)}{2}.
\]
Putting both of these together, we must solve the two inequalities
\begin{align}
2a-1 & \geq\tanh\left(\alpha\one+3\beta a^{2}\right)\nonumber \\
2b-1 & \leq\tanh\left(\alpha\one+3\beta b^{2}\right).\label{eq:uniqueness_inequalities}
\end{align}
By assumption, there is exactly one solution $x_{0}$ to the equation
$2x-1=\tanh\left(\alpha\one+3\beta x^{2}\right)$. By equation (\ref{eq:uniqueness_inequalities}),
we would then need that $a\geq x_{0}$ and $b\leq x_{0}$. But $a$
is the minimum off-diagonal entry of $X$ and $b$ is the maximum
off-diagonal entry of $X$, so they must be equal. Hence the constant
solution $x_{0}\one$ of Lemma \ref{lem:simple_constant_solution_D_bound}
is the only solution. See Figure \ref{fig:min_max} for an illustration.

\begin{figure}[H]
\begin{centering}
\includegraphics[scale=0.5]{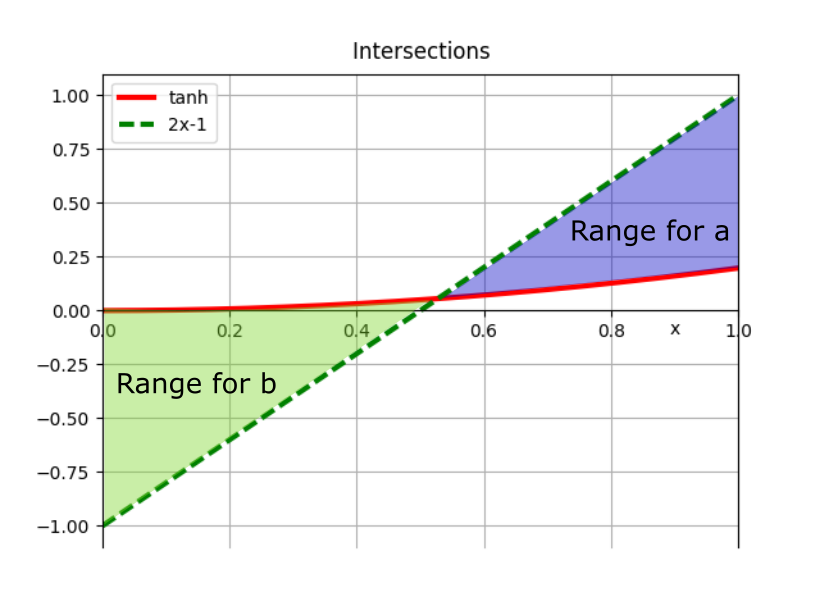}
\par\end{centering}
\caption{An illustration of the permissible range for $a$ and $b$.\label{fig:min_max}}
\end{figure}
In order to generalize this argument to any subgraph-counting function,
recall that every entry of $\grad f\left(x\one\right)$ is just some
polynomial $p\left(x\right)$. If all the weights are $\beta_{i}$
are positive then the preceding argument can be repeated for $p\left(x\right)$
with the inequalities all intact.

\subsection{Closeness\label{subsec:general_closeness}}
\begin{proof}[Proof of Theorem \ref{thm:positive_uniqueness}]
Let $X\in\cal X_{f}$. We would have liked to use an argument in
the same vein as that of subsection \ref{subsec:general_uniqueness}
and claim that the solution $X$ is close to a constant solution because
its minimum and maximum entries are close to each other. However,
this is not in general true: A matrix $X$ can easily have $\min X=0$
and $\max X=1$ while still satisfying the equation $\onenorm{X-\left(\one+\tanh\left(\grad f\left(X\right)\right)\right)/2}=o\left(n\right)$,
since the equation is not sensitive to changes in a small number of
entries. 

To overcome this, we will iterate the function $\frac{\one+\tanh\left(\grad f\left(X\right)\right)}{2}$,
showing that each time we do so, the minimum and maximum values tend
closer to a constant. 

Define the sequence of functions $\left\{ \varphi_{i}\right\} _{i=1}^{\infty}$
by $\varphi_{1}\left(x\right)=\varphi_{\boldsymbol{\beta}}\left(x\right)$
and $\varphi_{i+1}\left(x\right)=\varphi\left(\varphi_{i}\left(x\right)\right)$
for $i\geq1$. Denote $k=\ceil{\frac{\log\lambda}{\log D_{\boldsymbol{\beta}}}}=\ceil{\log_{D_{\boldsymbol{\beta}}}\left(\lambda\right)}$.
By assumption, for all $x_{0}\in\left[0,1\right]$ we have 
\[
\abs{\varphi_{\boldsymbol{\beta}}\left(x\right)-x_{0}}\leq D_{\boldsymbol{\beta}}\abs{x-x_{0}}.
\]
This implies that 
\begin{equation}
\abs{\varphi_{k}\left(x\right)-x_{0}}\leq D_{\boldsymbol{\beta}}^{k}\abs{x-x_{0}}\leq\lambda.\label{eq:h_is_close_to_x0}
\end{equation}
Denote by $\Phi:\r^{n}\rightarrow\r^{n}$ the function $\Phi\left(X\right)=\frac{\one+\tanh\left(\grad f\left(X\right)\right)}{2}$,
let $Y_{0}=X$ and recursively define $Y_{i+1}=\Phi\left(Y_{i}\right)$.
Then $\norm{Y_{k}-x_{0}\one}_{\infty}\leq\lambda$. To see this, observe
that since all $\beta$'s are positive, 
\begin{align*}
\min Y_{1}=\min\Phi\left(X\right) & =\min\frac{1+\tanh\left(\grad f\left(X\right)\right)}{2}\\
 & \geq\min\frac{1+\tanh\left(\grad f\left(\left(\min X\right)\one\right)\right)}{2}\\
 & =\varphi\left(\min X\right).
\end{align*}
Iterating, we have that 
\[
\min Y_{k}\geq\varphi_{k}\left(\min X\right).
\]
But by equation (\ref{eq:h_is_close_to_x0}), $\abs{\varphi_{k}\left(x\right)-x_{0}}<\lambda$
for every $x\in\left[0,1\right]$, and in particular for $\min X$.
Hence 
\[
\min Y_{k}\in\left[x_{0}-\lambda,x_{0}+\lambda\right].
\]
The same argument can be applied to $\max Y_{k}$, showing that all
of $Y_{k}$'s entries are in $\left[x_{0}-\lambda,x_{0}+\lambda\right]$.
Consequently, 
\begin{equation}
\onenorm{Y_{k}-x_{0}\one}\leq\lambda n.\label{eq:positive_initial_step-2}
\end{equation}
The distance between $X$ and $Y_{k}$ can be bounded as follows.
By Lemma \ref{lem:grad_is_lipschitz}, we have that for any two matrices
$A$ and $B$,
\begin{align*}
\onenorm{\Phi\left(A\right)-\Phi\left(B\right)} & \leq C_{\boldsymbol{\beta}}\onenorm{A-B},
\end{align*}
This gives a bound on consecutive iterations:
\begin{align*}
\onenorm{Y_{i}-Y_{i-1}} & =\onenorm{\Phi\left(Y_{i-1}\right)-\Phi\left(Y_{i-2}\right)}\\
 & \leq C_{\boldsymbol{\beta}}\onenorm{Y_{i-1}-Y_{i-2}},
\end{align*}
and so by induction, 
\[
\onenorm{Y_{i}-Y_{i-1}}\leq C_{\boldsymbol{\beta}}^{i}\onenorm{X-Y_{1}}=C_{\boldsymbol{\beta}}^{i}\onenorm{X-\Phi\left(X\right)}.
\]
Using this bound, we have
\begin{align}
\onenorm{X-Y_{k}} & =\onenorm{\sum_{i=1}^{k}Y_{i}-Y_{i-1}}\nonumber \\
 & \leq\sum_{i=1}^{k}\onenorm{Y_{i}-Y_{i-1}}\nonumber \\
 & \leq\sum_{i=1}^{k}C_{\boldsymbol{\beta}}^{i}\onenorm{X-\Phi\left(X\right)}\nonumber \\
 & \leq2C_{\boldsymbol{\beta}}^{k}\onenorm{X-\Phi\left(X\right)}.\label{eq:positive_final_step-1}
\end{align}
Combining equations (\ref{eq:positive_initial_step-2}), (\ref{eq:positive_final_step-1}),
and Theorem \ref{thm:main_theorem}, we have
\begin{align*}
\onenorm{X-x_{0}\one} & \leq\lambda n+C_{\boldsymbol{\beta}}^{\frac{\log\lambda}{\log D_{\boldsymbol{\beta}}}+1}10000C_{\boldsymbol{\beta}}^2 n^{15/16}\\
 & =\lambda n+10000C_{\boldsymbol{\beta}}^3\lambda^{\frac{\log C_{\boldsymbol{\beta}}}{\log D_{\boldsymbol{\beta}}}} n^{15/16}.
\end{align*}
Optimizing over $\lambda$ gives the dependence described in equation
(\ref{eq:optimized_lambda}).
\end{proof}

\section{Small weights}\label{sec:small_weights}

In this section we prove Theorem \ref{thm:small_weights}. 

\begin{proof}

We'll show that the function
\[
\Phi_{f}\left(X\right)=\frac{1+\tanh\left(\grad f\left(X\right)\right)}{2}
\]
is contracting if $S_{\boldsymbol{\beta}}<1$. For that, we'll need the following lemma, whose proof is postponed to the appendix:

\begin{lem}
	\label{lem:subgraph_counting_can_contract}Let $f$ be a subgraph
	counting function. Then for any two matrices $X,Y\in\overline{\mathcal{C}_{n}}$,
	\[
	\norm{\grad f\left(X\right)-\grad f\left(Y\right)}_{1}\leq\sum_{i=1}^{\ell}\abs{\beta_{i}}m_{i}\left(m_{i}-1\right)\onenorm{X-Y}.
	\]
\end{lem}

Using this lemma, we have that 

\begin{eqnarray}
\norm{\Phi_{f}\left(X\right)-\Phi_{f}\left(Y\right)}_{1} & = & \norm{\frac{\one+\tanh\left(\grad f\left(X\right)\right)}{2}-\frac{\one+\tanh\left(\grad f\left(Y\right)\right)}{2}}_{1}\nonumber \\
\left(\text{by equation (\ref{eq:tanh_is_contracting})}\right) & \leq & \frac{1}{2}\norm{\grad f\left(X\right)-\grad f\left(Y\right)}_{1}\nonumber \\
\left(\text{by Lemma \ref{lem:subgraph_counting_can_contract}}\right) & \leq & \frac{1}{2}\sum_{i=1}^{\ell}\abs{\beta_{i}}m_{i}\left(m_{i}-1\right)\onenorm{X-Y}\nonumber \\
& = & \sum_{i=1}^{\ell}\abs{\beta_{i}}{m_{i} \choose 2}\onenorm{X-Y}\nonumber \\
& = & S_{\boldsymbol{\beta}}\onenorm{X-Y}.\label{eq:phi_is_contracting}
\end{eqnarray}

If $S_{\boldsymbol{\beta}}<1$ then $\Phi_{f}\left(X\right)$ is contracting,
and by Banach's fixed point theorem it has a unique fixed point in
the compact space of all matrices with entries in $\left[0,1\right]$;
we already know by Lemma \ref{lem:simple_constant_solution_D_bound}
that it is a constant solution $X_{c}=c\cdot\one$. This shows the
first part of Theorem \ref{thm:small_weights}. For the second
part, let $X\in\cal X_{f}$. Then by a simple calculation, 

\begin{align*}
\norm{X-X_{c}}_{1} & =\norm{X-\Phi_{f}\left(X\right)+\Phi_{f}\left(X\right)-X_{c}+\Phi_{f}\left(X_{c}\right)-\Phi_{f}\left(X_{c}\right)}_{1}\\
& \leq\norm{X-\Phi_{f}\left(X\right)}_{1}+\norm{\Phi_{f}\left(X\right)-\Phi_{f}\left(X_{c}\right)}_{1}+\norm{X_{c}-\Phi_{f}\left(X_{c}\right)}_{1}\\
& =\norm{X-\Phi_{f}\left(X\right)}_{1}+\norm{\Phi_{f}\left(X\right)-\Phi_{f}\left(X_{c}\right)}_{1}\\
\left(\text{by equation (\ref{eq:phi_is_contracting})}\right) & \leq\norm{X-\Phi_{f}\left(X\right)}_{1}+S_{\boldsymbol{\beta}}\norm{X-X_{c}}_{1}.
\end{align*}
Rearranging, we get the desired result:
\[
\onenorm{X-X_{c}}\leq\frac{\norm{X-\Phi_{f}\left(X\right)}_{1}}{1-S_{\boldsymbol{\beta}}}\leq\frac{5000C_{\boldsymbol{\beta}}^2}{1-S_{\boldsymbol{\beta}}}n^{15/16}.
\]
\end{proof}

\section{Two block model}\label{sec:two_block_model}
The proof of Theorem \ref{thm:two_block_model} is rather technical. It goes roughly as
follows: We assume that there exists a fixed point of the form 
\[
X=\alpha_{1}v_{1}v_{1}^{T}+\alpha_{2}v_{2}v_{2}^{T}-\mathrm{I}\left(\alpha_{1}+\alpha_{2}\right),
\]
where $v_{1}$ is the vector $\left(1,1,\ldots,1\right)$ whose entries
are all $1$, and $v_{2}$ is the vector $\left(-1,\ldots,-1,1,\ldots1\right)$
whose first $N/2$ entries are $-1$ and whose second $N/2$ entries
are $1$. From this assumption we arrive at pair of non-linear scalar
equations for $\alpha_{1}$ and $\alpha_{2}$; non-trivial solutions
of these equations guarantee a non-trivial block model for $X$. We
then show by direct calculation that for large enough $\abs{\beta}$,
such a solution does indeed exist. 

We postpone the proof to the appendix.

\section{Acknowledgments}

The first author is grateful to Sourav Chatterjee for inspiring him
to work on this topic and for an enlightening discussion. We thank
Miel Sharf for his advice on contraction and Amir Dembo and Yufei
Zhao for an insightful conversation. Finally, we thank the anonymous referees 
for spurring us to improve our results and for comments bettering the 
presentation of this work.

\bibliographystyle{plain}
\bibliography{large_bibliography}

\section{Appendix}
\begin{proof}[Proof of Lemma \ref{lem:johnson_linden_scalar}]
We'll show the proof only for the inequality $\left\langle g\left(v_{1}\right),g\left(v_{2}\right)\right\rangle -\left\langle v_{1},v_{2}\right\rangle \leq2\delta$;
the inequality $\left\langle v_{1},v_{2}\right\rangle -\left\langle g\left(v_{1}\right),g\left(v_{2}\right)\right\rangle \leq2\delta$
follows a similar calculation. 

The scalar product between any two vectors $x$ and $y$ can be written
as a function of $x+y$ and $x-y$:
\[
\left\langle x,y\right\rangle =\frac{1}{4}\left(\norm{x+y}^{2}-\norm{x-y}^{2}\right).
\]
We can now calculate: 
\begin{align*}
\left\langle g\left(v_{1}\right),g\left(v_{2}\right)\right\rangle  & =\frac{1}{4}\left(\norm{g\left(v_{1}\right)+g\left(v_{2}\right)}^{2}-\norm{g\left(v_{1}\right)-g\left(v_{2}\right)}^{2}\right)\\
 & =\frac{1}{4}\left(\norm{g\left(v_{1}+v_{2}\right)}^{2}-\norm{g\left(v_{1}-v_{2}\right)}^{2}\right)\\
 & \leq\frac{1}{4}\left(\left(1+\delta\right)\norm{v_{1}+v_{2}}^{2}-\left(1-\delta\right)\norm{v_{1}-v_{2}}^{2}\right)\\
 & =\frac{1}{4}\left(4\left\langle v_{1},v_{2}\right\rangle +\delta\norm{v_{1}+v_{2}}^{2}+\delta\norm{v_{1}-v_{2}}^{2}\right)\\
\left(\text{because }\norm{v_{1}\pm v_{2}}^{2}\leq4\right) & \leq\frac{1}{4}\left(4\left\langle v_{1},v_{2}\right\rangle +8\delta\right)\\
 & =\left\langle v_{1},v_{2}\right\rangle +2\delta.
\end{align*}
This implies that $\left\langle g\left(v_{1}\right),g\left(v_{2}\right)\right\rangle -\left\langle v_{1},v_{2}\right\rangle \leq2\delta$.
\end{proof}

\begin{lem}
	\label{lem:products_into_sums}Let $I\subseteq\left[n\right]$ be
	a set of indices. Then for any $X,Y\in\overline{\mathcal{C}_{n}}$, 
	\[
	\abs{\prod_{\alpha\in I}X_{\alpha}-\prod_{\alpha\in I}Y_{\alpha}}\leq\sum_{\alpha\in I}\abs{X_{\alpha}-Y_{\alpha}}.
	\]
\end{lem}

\begin{proof}
By induction on $\abs I$. Let $\beta\in I$. Then 
\begin{align*}
\abs{\prod_{\alpha\in I}X_{\alpha}-\prod_{\alpha\in I}Y_{\alpha}} & =\abs{\prod_{\substack{\alpha\in I}
	}X_{\alpha}-\prod_{\alpha\in I}Y_{\alpha}+Y_{\beta}\prod_{\substack{\alpha\in I\\
			\substack{\alpha\neq\beta}
		}
	}X_{\alpha}-Y_{\beta}\prod_{\substack{\alpha\in I\\
			\substack{\alpha\neq\beta}
		}
	}X_{\alpha}}\\
& =\abs{\left(\prod_{\substack{\alpha\in I\\
			\substack{\alpha\neq\beta}
		}
	}X_{\alpha}\right)\left(X_{\beta}-Y_{\beta}\right)+Y_{\beta}\left(\prod_{\substack{\alpha\in I\\
			\substack{\alpha\neq\beta}
		}
	}X_{\alpha}-\prod_{\substack{\alpha\in I\\
			\substack{\alpha\neq\beta}
		}
	}Y_{\alpha}\right)}\\
& \leq\abs{X_{\beta}-Y_{\beta}}+\abs{\prod_{\substack{\alpha\in I\\
			\substack{\alpha\neq\beta}
		}
	}X_{\alpha}-\prod_{\substack{\alpha\in I\\
			\substack{\alpha\neq\beta}
		}
	}Y_{\alpha}},
\end{align*}
where the last inequality is because $\abs{X_{\alpha}}\leq1$ for all $\alpha$.
\end{proof}

\begin{proof}[Proof of Lemma \ref{lem:subgraph_counting_can_contract}]

It is enough to show the result for a function $f$ that counts
just a single subgraph $H=\left(V,E\right)$ with $m:=\abs E$; the
general result follows by linearity of the derivative and the triangle
inequality. By equation (\ref{eq:how_to_calculate_grad_f}), 
\[
\partial f_{ij}\left(X\right)=\frac{\beta}{\left(N-2\right)\left(N-3\right)\ldots\left(N-m+1\right)}\sum_{\left(a,b\right)\in E}\sum_{\underset{q_{a}=i,q_{b}=j}{\underset{\text{\ensuremath{q} has distinct elements}}{q\in\left[N\right]^{m}}}}\prod_{\underset{\left\{ l,l'\right\} \neq\left\{ a,b\right\} }{\left(l,l'\right)\in E}}X_{q_{l},q_{l'}}.
\]
The difference between the gradients is then 
\begin{align*}
\norm{\grad f\left(X\right)-\grad f\left(Y\right)}_{1} & = \sum_{ij}\bigg\lvert \frac{\abs{\beta}}{\left(N-2\right)\left(N-3\right)\ldots\left(N-m+1\right)}\cdot\\
& \,\,\,\,\,\,\,\,\,\,\,\cdot\sum_{\left(a,b\right)\in E}\sum_{\underset{q_{a}=i,q_{b}=j}{\underset{\text{\ensuremath{q} has distinct elements}}{q\in\left[N\right]^{m}}}}\left(\prod_{\underset{\left\{ l,l'\right\} \neq\left\{ a,b\right\} }{\left(l,l'\right)\in E}}X_{q_{l},q_{l'}}-\prod_{\underset{\left\{ l,l'\right\} \neq\left\{ a,b\right\} }{\left(l,l'\right)\in E}}Y_{q_{l},q_{l'}}\right) \bigg\rvert.
\end{align*}
By Lemma \ref{lem:products_into_sums}, this can be bounded by 
\begin{align*}
\norm{\grad f\left(X\right)-\grad f\left(Y\right)}_{1} & \leq\sum_{ij}\frac{\abs{\beta}}{\left(N-2\right)\left(N-3\right)\ldots\left(N-m+1\right)}\cdot\\
& \,\,\,\,\,\,\,\,\,\,\,\cdot\sum_{\left(a,b\right)\in E}\sum_{\underset{q_{a}=i,q_{b}=j}{\underset{\text{\ensuremath{q} has distinct elements}}{q\in\left[N\right]^{m}}}}\sum_{\underset{\left\{ l,l'\right\} \neq\left\{ a,b\right\} }{\left(l,l'\right)\in E}}\abs{X_{q_{l},q_{l'}}-Y_{q_{l},q_{l'}}}.
\end{align*}
Fix a pair of vertices $\alpha,\beta$. By symmetry, as $i$ and $j$
span over all possible pairs of vertices, the term $\abs{X_{\alpha,\beta}-Y_{\alpha,\beta}}$
appears $m\left(m-1\right)\left(N-2\right)\left(N-3\right)\ldots\left(N-m+1\right)$ times. Thus 
\[
\norm{\grad f\left(X\right)-\grad f\left(Y\right)}_{1}\leq\abs{\beta}m\left(m-1\right)\onenorm{X-Y}.
\]
\end{proof}

\begin{lem}
\label{lem:unique_tanh_intersection}For every $\alpha\in\r$, the
equation 
\[
2x-1=\tanh\left(\alpha x^{2}\right)
\]
has a unique solution with $x\in\left(0,1\right)$.
\end{lem}

\begin{proof}
Denote $g\left(x\right)=2x-1$ and $h\left(x\right)=\tanh\left(\alpha x^{2}\right)$;
we must then show that there is a unique point $x\in\left(0,1\right)$
such that $g\left(x\right)=h\left(x\right)$.
\begin{itemize}
\item The case $\alpha=0$ is solved by $x=\frac{1}{2}$.
\item The case $\alpha<0$: The function $g\left(x\right)$ is strictly
increasing with $g\left(0\right)=-1$ and $g\left(1\right)=1$, while
$h\left(0\right)=0$ and is strictly decreasing. A solution exists
as both functions are continuous.
\item The case $\alpha>0$: The function $g\left(x\right)$ is increasing
with $g\left(0\right)=-1$ and $g\left(1\right)=1$, while $h\left(0\right)=0$
and $h$ is strictly bounded by $1$; hence by continuity a solution
exists. For uniqueness of this solution, denote the smallest point
of intersection of $g$ and $h$ by $x_{1}$. Note that $x_{1}>\frac{1}{2}$,
since $g\left(\frac{1}{2}\right)=0$ and $h\left(\frac{1}{2}\right)>0$.
Since $g\left(x\right)<h\left(x\right)$ in the interval $\left[0,x_{1}\right)$,
the derivative $h'$ must be no greater than $g'=2$ at $x_{1}$.
But in order for there to be another point of intersection, the derivative
must be larger than $2$ at some point in the interval $\left[x_{1},1\right]$.
Differentiating, we have
\begin{equation}
h'\left(x\right)=\frac{2\alpha x}{\cosh^{2}\left(\alpha x^{2}\right)}.\label{eq:h_deriv}
\end{equation}
Differentiating again, we have
\[
h''\left(x\right)=\frac{2\alpha}{\cosh^{2}\left(\alpha x^{2}\right)}\left(1-4\alpha x^{2}\tanh\left(\alpha x^{2}\right)\right).
\]
The maximum of the derivative is attained when the second derivative
is $0$, that is, $1-4\alpha x^{2}\tanh\left(\alpha x^{2}\right)=0$.
This implies that $2\alpha x=\frac{1}{2\tanh\left(\alpha x^{2}\right)x}$.
Substituting this into equation (\ref{eq:h_deriv}), we get that for
all $\frac{1}{2}<x<1$,
\begin{align*}
h'\left(x\right) & =\frac{2\alpha x}{\cosh^{2}\left(\alpha x^{2}\right)}\\
 & \leq\frac{1}{x\cdot2\tanh\left(\alpha x^{2}\right)\cosh^{2}\left(\alpha x^{2}\right)}\\
 & =\frac{1}{x\cdot2\sinh\left(\alpha x^{2}\right)\cosh\left(\alpha x^{2}\right)}\\
 & =\frac{1}{x\cdot\sinh\left(2\alpha x^{2}\right)}<2
\end{align*}
since $x>\frac{1}{2}$ and $\sinh\left(2\alpha x^{2}\right)>1$. Hence
no other intersection point exists.
\end{itemize}
See Figure \ref{fig:tanh_intersections} for a visual illustration
of $g$ and $h$.

\begin{figure}[H]
\begin{centering}
\includegraphics[scale=0.5]{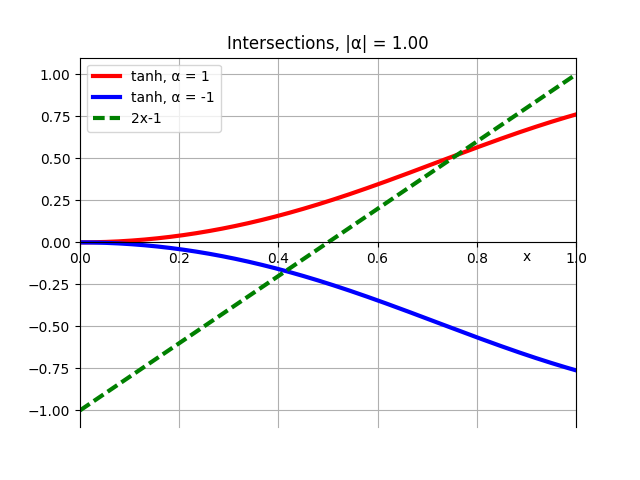}\includegraphics[scale=0.5]{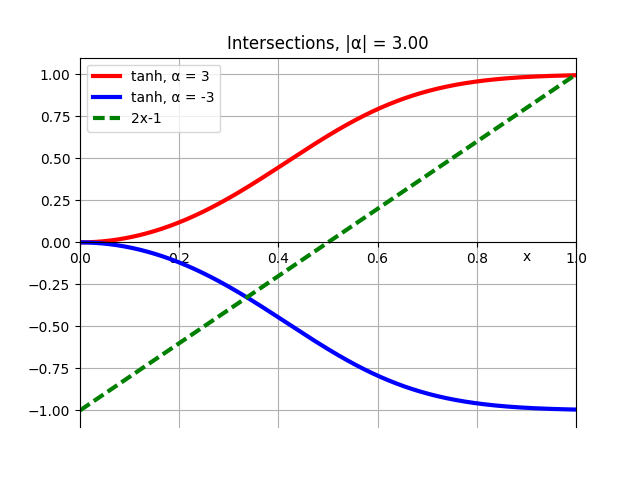}
\par\end{centering}
\caption{Two examples showing that there is only one intersection between $2x-1$
and $\tanh\left(\alpha x^{2}\right)$. \label{fig:tanh_intersections}}
\end{figure}
\end{proof}
\begin{proof}[Proof of Theorem \ref{thm:two_block_model}]
For simplicity, instead of solving the equation $X=\frac{\one+\tanh\left(\frac{3\beta}{N-2}\overline{X^{2}}\right)}{2}$
for negative $\beta$, we will solve the equation $X=\frac{\one-\tanh\left(\frac{\beta}{N-2}\overline{X^{2}}\right)}{2}$
for positive $\beta$ (where we assimilated the factor of $3$ inside
$\beta$). 

Denote by $v_{1}$ the vector $\left(1,1,\ldots,1\right)$ whose entries
are all $1$, and by $v_{2}$ the vector $\left(-1,\ldots,-1,1,\ldots1\right)$
whose first $N/2$ entries are $-1$ and whose second $N/2$ entries
are $1$. Let 
\[
X=\alpha_{1}v_{1}v_{1}^{T}+\alpha_{2}v_{2}v_{2}^{T}-\mathrm{I}\left(\alpha_{1}+\alpha_{2}\right).
\]
Then $X$ is a symmetric matrix with $0$ on the diagonal, $\alpha_{1}+\alpha_{2}$
in the top left and bottom right quarters, and $\alpha_{1}-\alpha_{2}$
in the top right and bottom left quarters. Squaring $X$, we get
\begin{align*}
X^{2} & =\left(\alpha_{1}v_{1}v_{1}^{T}+\alpha_{2}v_{2}v_{2}^{T}-\mathrm{I}\left(\alpha_{1}+\alpha_{2}\right)\right)^{2}\\
 & =\alpha_{1}^{2}\left(v_{1}v_{1}^{T}\right)^{2}+\alpha_{2}\left(v_{2}v_{2}^{T}\right)^{2}+\mathrm{I}\left(\alpha_{1}+\alpha_{2}\right)^{2}-2\alpha_{1}\left(\alpha_{1}+\alpha_{2}\right)v_{1}v_{1}^{T}-2\alpha_{2}\left(\alpha_{1}+\alpha_{2}\right)v_{2}v_{2}^{T}\\
 & =\left(\alpha_{1}^{2}\left(N-2\right)-2\alpha_{1}\alpha_{2}\right)v_{1}v_{1}^{T}+\left(\alpha_{2}^{2}\left(N-2\right)-2\alpha_{1}\alpha_{2}\right)v_{2}v_{2}^{T}+\mathrm{I}\left(\alpha_{1}+\alpha_{2}\right)^{2}.
\end{align*}
Setting the diagonal to zero, we have
\[
\overline{X^{2}}=\left(\alpha_{1}^{2}\left(N-2\right)-2\alpha_{1}\alpha_{2}\right)v_{1}v_{1}^{T}+\left(\alpha_{2}^{2}\left(N-2\right)-2\alpha_{1}\alpha_{2}\right)v_{2}v_{2}^{T}-\mathrm{I}\left(\alpha_{1}^{2}\left(N-2\right)+\alpha_{2}^{2}\left(N-2\right)-4\alpha_{1}\alpha_{2}\right).
\]
So $\frac{\beta}{N-2}\overline{X^{2}}$ is a symmetric matrix with
$0$ on the diagonal, $\frac{\beta}{N-2}\left(\left(\alpha_{1}^{2}+\alpha_{2}^{2}\right)\left(N-2\right)-4\alpha_{1}\alpha_{2}\right)$
in the top left and bottom right quarters, and $\beta\left(\alpha_{1}^{2}-\alpha_{2}^{2}\right)$
in the top right and bottom left quarters. The matrix $\tanh\left(\frac{\beta}{N-2}\overline{X^{2}}\right)$
can then also be written as a sum of the form $av_{1}v_{1}^{T}+bv_{2}v_{2}^{T}-\mathrm{I}\left(a+b\right)$,
where
\begin{align}
\tanh\left(\beta\left(\left(\alpha_{1}^{2}+\alpha_{2}^{2}\right)\left(N-2\right)-\frac{4}{N-2}\alpha_{1}\alpha_{2}\right)\right) & =a+b\nonumber \\
\tanh\left(\beta\left(\alpha_{1}^{2}-\alpha_{2}^{2}\right)\right) & =a-b.\label{eq:alphas_and_tanhs}
\end{align}
The expression$\frac{\one-\tanh\left(\frac{\beta}{N-2}\overline{X^{2}}\right)}{2}$
can then be written as
\[
\frac{\one-\tanh\left(\frac{\beta}{N-2}\overline{X^{2}}\right)}{2}=\frac{1-a}{2}v_{1}v_{1}^{T}-\frac{b}{2}v_{2}v_{2}^{T}+\mathrm{I}\left(\frac{1}{2}a+\frac{1}{2}b-\frac{1}{2}\right).
\]
Equating this with $X$, we get
\begin{align*}
\alpha_{1} & =\frac{1-a}{2}\\
\alpha_{2} & =-\frac{b}{2}.
\end{align*}
Rearranging and plugging into equation (\ref{eq:alphas_and_tanhs}),
we obtain the following two equations in two variables:
\begin{align}
\tanh\left(\beta\left(\left(\alpha_{1}^{2}+\alpha_{2}^{2}\right)-\frac{4}{N-2}\alpha_{1}\alpha_{2}\right)\right) & =1-2\alpha_{1}-2\alpha_{2}\nonumber \\
\tanh\left(\beta\left(\alpha_{1}^{2}-\alpha_{2}^{2}\right)\right) & =1-2\alpha_{1}+2\alpha_{2}.\label{eq:main_tanh_equations_alpha}
\end{align}
We will now show that for large enough $\beta$, these equations have
at least two solutions. As shown in Lemma \ref{lem:simple_constant_solution_D_bound},
there is always a constant $X=c\cdot\one$ is solution to the fixed
point equation (\ref{eq:fixed_point_equation}). It corresponds to
the case $\alpha_{2}=0$; in this case the two equations both identify
to $\tanh\left(\beta\alpha_{1}^{2}\right)=1-2\alpha_{1}.$ We must
therefore show that that for large enough $\beta$, there is a solution
with $\alpha_{2}\neq0$.

Let us change variables in order to bring the equations to a more
friendly form. Denote $x=\alpha_{1}+\alpha_{2}$ and $y=\alpha_{1}-\alpha_{2}$.
Then $\alpha_{1}^{2}-\alpha_{2}^{2}=xy$, $\alpha_{1}^{2}+\alpha_{2}^{2}=\frac{1}{2}\left(x^{2}+y^{2}\right)$
and $\alpha_{1}\alpha_{2}=\frac{1}{4}\left(x^{2}-y^{2}\right)$, and
(\ref{eq:main_tanh_equations_alpha}) can be rewritten as 
\begin{align}
\tanh\left(\frac{\beta}{N-2}\left(\frac{N-4}{2}x^{2}+\frac{N}{2}y^{2}\right)\right) & =1-2x\nonumber \\
\tanh\left(\beta xy\right) & =1-2y.\label{eq:main_tanh_equations_xy}
\end{align}
We now need to show that there exists a solution with $x\neq y$. 

The matrix $X$ has entries in $\left[0,1\right]$, so we know that

\begin{align*}
0 & \leq\alpha_{1}-\alpha_{2}\leq1\\
0 & \leq\alpha_{1}+\alpha_{2}\leq1.
\end{align*}
Hence $x$ and $y$ are also in $\left[0,1\right]$. For the first
equation in (\ref{eq:main_tanh_equations_xy}), if $x$ is small enough,
then there is a unique $y\in\r$ the satisfies it. Denote this $y$
by $g\left(x\right)$; its range and domain will be calculated later.
For the second equation, a unique $y\in\left(0,1\right)$ exists for
all $x\in\left[0,1\right]$ since $\tanh\left(\beta xy\right)$ is
an increasing function of $y$ while $1-2y$ is a decreasing function
$y$. Denote this $y$ by $h\left(x\right):\left[0,1\right]\rightarrow\left(0,1\right)$.

Showing that a non-constant solution exists therefore requires showing
that $g$ and $h$ intersect at a point for which $x\neq y$. Figure
\ref{fig:g_h_intersections} shows that this is indeed the case for
large enough $\beta$ (by numerical calculations, the solution first
appears at around $\beta\approx22$, if we approximate $N-4\approx N-2\approx N$).

\begin{figure}[H]
\begin{centering}
\includegraphics[scale=0.5]{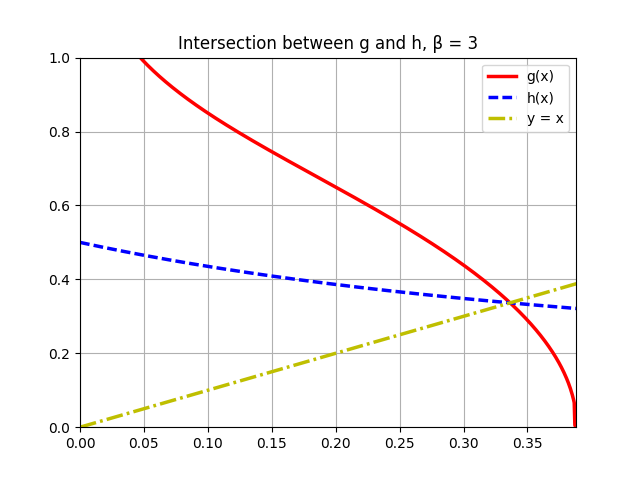}\includegraphics[scale=0.5]{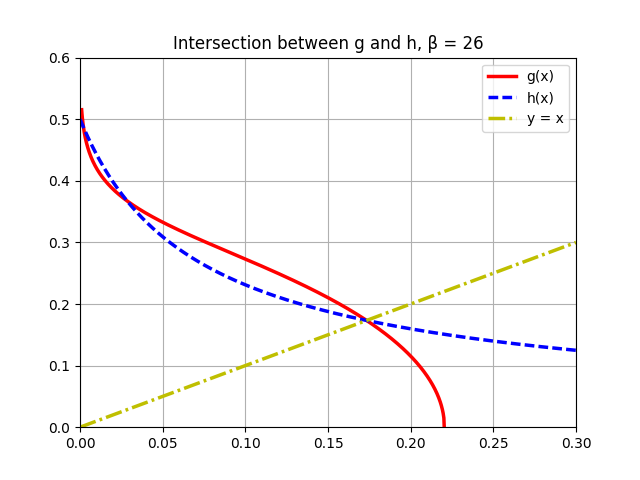}
\par\end{centering}
\caption{Two examples illustrating the behavior of $g$ and $h$. For small
$\beta$ there is only one intersection between them, while for large
$\beta$ there are $3$. For making these images, $N$ was assumed
large enough so that $N-4\approx N-2\approx N$. \label{fig:g_h_intersections}}
\end{figure}

Let us now grit our teeth and show this result analytically. First
consider $h$. It satisfies the functional equation
\[
\tanh\left(\beta xh\left(x\right)\right)-1+2h\left(x\right)=0.
\]
At $x=0$, we must have $h\left(0\right)=\frac{1}{2}$. Differentiating,
we get

\begin{align*}
\frac{\beta\left(h\left(x\right)+xh'\left(x\right)\right)}{\cosh^{2}\left(\beta xh\left(x\right)\right)}+2h'\left(x\right) & =0.
\end{align*}
Isolating $h'$, we obtain
\[
h'\left(x\right)=-\frac{\beta h\left(x\right)}{\beta x+2\cosh^{2}\left(\beta xh\left(x\right)\right)}.
\]
Thus $h$ is decreasing. Forgoing calculations, differentiating again
shows that $h''$ is positive. Hence $h'$ is increasing, so we can
bound $h'$ by 
\begin{align}
h'\left(x\right) & \geq h'\left(0\right)\nonumber \\
 & =-\frac{\beta h\left(0\right)}{\beta\cdot0+2\cosh^{2}\left(\beta\cdot0\cdot h\left(0\right)\right)}\nonumber \\
 & =-\beta/2.\label{eq:derivative_of_h_is_bounded}
\end{align}
Now consider $g$. It satisfies the functional equation 

\begin{equation}
\tanh\left(\frac{\beta}{N-2}\left(\frac{N-4}{2}x^{2}+\frac{N}{2}g^{2}\left(x\right)\right)\right)-1+2x=0\label{eq:g_functional_equation}
\end{equation}
First let us calculate its domain. 

There exists an $x_{1}>0$ such that $g\left(x_{1}\right)=1$. Indeed,
setting $g\left(x\right)=1$, we have
\[
\tanh\left(\frac{\beta}{2}\left(\frac{N-4}{N-2}x^{2}+\frac{N}{N-2}\right)\right)=1-2x.
\]
At $x=0$, the left hand side is equal to $\tanh\left(\frac{\beta}{2}\frac{N}{N-2}\right)$,
which is smaller than $1$. The left hand side is increasing as a
function of $x$, while the right hand side is decreasing as a function
of $x$, with derivative $-2$. Hence a solution $x_{1}$ exists,
with 
\[
x_{1}\leq\frac{1-\tanh\left(\frac{\beta}{2}\right)}{2}.
\]
Using $\tanh\left(z\right)=\frac{1-e^{-2z}}{1+e^{2z}}$, this can
also be written as
\[
x_{1}\leq\frac{1-\frac{1-e^{-\beta}}{1+e^{-\beta}}}{2}=\frac{\frac{2e^{-\beta}}{1+e^{-\beta}}}{2}=\frac{2e^{-\beta}}{1+e^{-\beta}}\leq2e^{-\beta}.
\]
There exists an $x_{2}$ such that $g\left(x_{2}\right)=0$. Indeed,
setting $g\left(x\right)=0$, we get
\[
\tanh\left(\beta\frac{N-4}{2\left(N-2\right)}x^{2}\right)=1-2x,
\]
and a unique solution exists by Lemma \ref{lem:unique_tanh_intersection}.
It is clear that for all $x_{1}<x<x_{2}$, a unique solution exists
for $g\left(x\right)$. Differentiating equation (\ref{eq:g_functional_equation}),
we get

\[
\frac{\frac{\beta}{N-2}\left(\left(N-4\right)x+Ng\left(x\right)g'\left(x\right)\right)}{\cosh^{2}\left(\frac{\beta}{N-2}\left(\frac{N-4}{2}x^{2}+\frac{N}{2}g^{2}\left(x\right)\right)\right)}+2=0,
\]
and isolating $g'$, we obtain
\[
g'\left(x\right)=\frac{-2\cosh^{2}\left(\frac{\beta}{N-2}\left(\frac{N-4}{2}x^{2}+\frac{N}{2}g^{2}\left(x\right)\right)\right)-\frac{\beta}{N-2}\left(N-4\right)x}{\beta g\left(x\right)}.
\]
This is negative, and so $g$ is decreasing. The domain of $g$ is
therefore $\left[x_{1},x_{2}\right]$, and its range is $\left[0,1\right]$.

We may now finally inspect the intersection of $g$ and $h$. Let
$\eps=\frac{2}{\beta^{2}}$, and let $\beta$ be large enough so that
$\frac{1}{2}\eps=\frac{1}{\beta^{2}}>2e^{-\beta}>x_{1}$; this implies
that $\eps-x_{1}\geq\frac{1}{\beta^{2}}$. By (\ref{eq:derivative_of_h_is_bounded})
and the fact that $h\left(0\right)=\frac{1}{2}$, we have that 
\begin{align*}
h\left(\eps\right) & \geq\frac{1}{2}-\eps\beta/2\\
 & =\frac{1}{2}-\frac{1}{\beta}.
\end{align*}
Assume by contradiction that in the interval $\left[x_{1},\eps\right]$,
there is no intersection between $g$ and $h$. Since $g\left(x_{1}\right)=1>h\left(x_{1}\right)$,
this means that in $g\left(x\right)>h\left(x\right)$ for the entire
interval $\left[x_{1},\eps\right]$. In particular we have $g\left(x\right)>\frac{1}{2}-\frac{1}{\beta}$.
We can then give a bound on the derivative $g'$:
\begin{align*}
g'\left(x\right) & =\frac{-2\cosh^{2}\left(\frac{\beta}{N-2}\left(\frac{N-4}{2}x^{2}+\frac{N}{2}g^{2}\left(x\right)\right)\right)-\frac{\beta}{N-2}\left(N-4\right)x}{\beta g\left(x\right)}\\
 & \leq\frac{-2\cosh^{2}\left(\beta g^{2}\left(x\right)\right)}{\beta}\\
 & \leq\frac{-2\cosh^{2}\left(\beta\left(\frac{1}{2}-\frac{1}{\beta}\right)^{2}\right)}{\beta}\\
\left(\text{for \ensuremath{\beta>4}}\right) & \leq\frac{-2\cosh^{2}\left(\beta\left(\frac{1}{4}\right)^{2}\right)}{\beta}\\
 & =\frac{-2\cosh^{2}\left(\frac{\beta}{16}\right)}{\beta}.
\end{align*}
We then have 
\begin{align*}
g\left(\eps\right) & \leq g\left(x_{1}\right)+\left(\eps-x_{1}\right)\frac{-2\cosh^{2}\left(\frac{\beta}{16}\right)}{\beta}\\
 & \leq1+\frac{1}{\beta^{2}}\frac{-2\cosh^{2}\left(\frac{\beta}{16}\right)}{\beta}\\
 & =1-\frac{2\cosh^{2}\left(\frac{\beta}{16}\right)}{\beta^{3}}.
\end{align*}
This quantity goes to $-\infty$ as $\beta\goinf$. This is a contradiction,
as we assumed $g\left(x\right)\geq\frac{1}{2}-\frac{1}{\beta}$ in
the interval $\left[x_{1},\eps\right]$. Thus for $\beta$ large enough,
the curves $g$ and $h$ intersect at a point $x^{*}\in\left[x_{1},\frac{2}{\beta^{2}}\right]$.
This intersection point satisfies $g\left(x^{*}\right)\geq\frac{1}{2}-\frac{1}{\beta}$;
for $\beta>4$, we have $y=g\left(x^{*}\right)>\frac{1}{4}$. However
$x^{*}\leq\frac{2}{\beta^{2}}<\frac{1}{4}$. This intersection point
does not satisfy $x=y$ and therefore does not correspond to the constant
solution. 

Finally, as $\beta\goinf$, it is clear that $x^{*}\rightarrow0$
and $y^{*}=g\left(x^{*}\right)\rightarrow\frac{1}{2}$, implying that
$\alpha_{1}\rightarrow\frac{1}{4}$ and $\alpha_{2}\rightarrow-\frac{1}{4}$,
meaning that $X$ tends to the adjacency matrix of a bipartite graph.
\end{proof}

\end{document}